\documentclass[12pt]{article}
\usepackage{amscd, amssymb, amsmath, calc, enumerate}
\begin{document}
\newenvironment{eq}{\begin{equation}}{\end{equation}}
\newenvironment{proof}{{\bf Proof}:}{\vskip 5mm }
\newenvironment{rem}{{\bf Remark}:}{\vskip 5mm }
\newenvironment{remarks}{{\bf Remarks}:\begin{enumerate}}{\end{enumerate}}
\newenvironment{examples}{{\bf Examples}:\begin{enumerate}}{\end{enumerate}}  
\newtheorem{proposition}{Proposition}[subsection]
\newtheorem{piece}[proposition]{}
\newtheorem{lemma}[proposition]{Lemma}
\newtheorem{definition}[proposition]{Definition}
\newtheorem{theorem}[proposition]{Theorem}
\newtheorem{cor}[proposition]{Corollary}
\newtheorem{conjecture}{Conjecture}
\newtheorem{pretheorem}[proposition]{Pretheorem}
\newtheorem{hypothesis}[proposition]{Hypothesis}
\newtheorem{example}[proposition]{Example}
\newtheorem{remark}[proposition]{Remark}
\newtheorem{ex}[proposition]{Exercise}
\newtheorem{cond}[proposition]{Conditions}
\newtheorem{cons}[proposition]{Construction}
\newcommand{\llabel}[1]{\label{#1}}
\newcommand{\comment}[1]{}
\newcommand{\sr}{\rightarrow}
\newcommand{\dw}{\downarrow}
\newcommand{\bdl}{\bar{\Delta}}
\newcommand{\zz}{{\bf Z\rm}}
\newcommand{\zq}{{\bf Z}_{qfh}}
\newcommand{\nn}{{\bf N\rm}}
\newcommand{\qq}{{\bf Q\rm}}
\newcommand{\nq}{{\bf N}_{qfh}}
\newcommand{\oo}{\otimes}
\newcommand{\uu}{\underline}
\newcommand{\ih}{\uu{Hom}}
\newcommand{\af}{{\bf A}^1}
\newcommand{\dsr}{\stackrel{\sr}{\scriptstyle\sr}}
\newcommand{\PP}{$P_{\infty}$}
\newcommand{\DD}{D}
\newcommand{\tp}{\tilde{D}}
\newcommand{\HH}{$H_{\infty}$}
\newcommand{\ii}{\stackrel{\scriptstyle\sim}{\sr}}
\newcommand{\BB}{_{\bullet}}

\begin{center}
{\Large\bf Lectures on motivic cohomology 2000/2001}\\ 
\vskip 2mm
{\Large\bf (written by
Pierre Deligne)}\\
\vskip 4mm {\large\bf by Vladimir Voevodsky}\footnote{Supported by the
NSF grants DMS-97-29992 and DMS-9901219 and The Ambrose Monell
Foundation}$^,$\footnote{School of Mathematics, Institute for Advanced
Study, Princeton NJ, USA. e-mail: vladimir@ias.edu}\\
\end{center}
\vskip 4mm
\tableofcontents

\numberwithin{equation}{proposition}

%\section{Lecture 1}
%
\section{Introduction}

The lectures which provided the source for these notes covered several
different topics which are related to each other but which do not in
any reasonable sense form a coherent whole. As a result, this text is
a collection of four parts which refer to each other but otherwise are
independent.

In the first part we introduce the motivic homotopy category and
connect it with the motivic cohomology theory discussed in
\cite{CC}. The exposition is a little unusual because we wanted to
avoid any references to model structures and still prove the main
theorem \ref{1.6}. We were able to do it modulo \ref{6.3} where we had
to refer to the next part.

The second part is about we the motivic homotopy category of
$G$-schemes where $G$ is a finite flat group scheme with respect to an
equivariant analog of the Nisnevich topology. Our main result is a
description of the class of $\af$-equivalences (formerly called
$\af$-weak equivalences) given in \ref{postponed} (also in
\ref{conv}). For the trivial group $G$ we get a new description of the
$\af$-equivalences in the non equivariant setting. Most of the
material of this part can also be found in \cite{HH0} and \cite{HH2}.

In the third part we define a class of sheaves on $G$-schemes which
we call solid sheaves. It contains all representable sheaves and
quotients of representable sheaves by subsheaves corresponding to open
subschemes. In particular the Thom spaces of vector bundles are solid
sheaves. The key property of solid sheaves can be expressed by saying
that any right exact functor which takes open embeddings to
monomorphisms is left exact on solid sheaves. A more precise statement
is \ref{15.1}. 

In the fourth part we study two functors. One is the extension to
pointed sheaves of the functor from $G$-schemes to schemes which takes
$X$ to $X/G$. The other one is extension to pointed sheaves of the
functor which takes $X$ to $X^W$ where $W$ is a finite flat
$G$-scheme. We show that both functors take solid sheaves to
solid sheaves and preserve local and $\af$-equivalences between
termwise (ind-)solid sheaves.

The material of all the parts of these notes but the first one was
originally developed with one particular goal in mind - to extend
non-additive functors, such as the symmetric product, from schemes to
the motivic homotopy category. More precisely, we were interested in
functors given by
$$T:X\mapsto (X^W\times E)/G$$
where $G$ is a finite flat group scheme, $W$ is a finite flat
$G$-scheme and $E$ any $G$-scheme of finite type.  The equivariant
motivic homotopy category was introduced to represent $T$ as a
composition
$$X\mapsto X^W\mapsto X^W\times E\mapsto (X^W\times E)/G$$
and solid sheaves as a natural class of sheaves on which the derived
functor ${\bf L}T$ coincides with $T$. 

In the present form these notes are the result of an interactive
process which involved all listeners of the lectures.  A very special
role was played by Pierre Deligne. The text as it is now was
completely written by him. He also cleared up a lot of messy parts and
simplified the arguments in several important places.

\comment{There is no exact correspondence between sections in the notes the
lectures. The actual contents of individual lectures is given by the
following list:
\begin{description}
\item[Lecture 1] Sections \ref{2}, \ref{3}, Section \ref{4} -
Proposition \ref{1.7} with the proof, Section \ref{5} up to
(\ref{chain})
\item[Lecture 2] Section \ref{4} - Proposition \ref{1.8} with the
proof, Section \ref{5} up to the formulation of Proposition
\ref{3p1}, Section \ref{5.5} up to Proposition \ref{3p2}
\item[Lecture 3] Section \ref{4} - from Proposition \ref{3.3} to
Corollary \ref{c3.6}, Section \ref{5} - proof of Proposition
\ref{3p1}, Section \ref{5.5} - Proposition \ref{3p2}, Section
\ref{6}.
\item[Lecture 4] Section \ref{7}
\item[Lecture 5] Section \ref{8}
\item[Lecture 6] Section \ref{10}
\item[Lecture 7] Section \ref{11} 
\item[Lecture 8] ----
\item[Lecture 9] Section \ref{12}
\end{description}}

\section{Motivic cohomology and motivic homotopy category}

We will recall first some of last year results (see \cite{CC}).

\subsection{Last year}\label{2}
\paragraph{1.1} We work over a field $k$ which sometimes will have to
be assumed to be perfect. The schemes over $k$ we consider will
usually be assumes separated and smooth of finite type over $k$. We
note $Sm/k$ their category. Three Grothendieck topologies on $Sm/k$
will be useful: Zariski, Nisnevich and etale. For each of these
topologies a sheaf on $(Sm/k)$ amounts to the data for $X$ smooth over
$k$ of a sheaf $F_{X}$ on the small site $X_{Zar}$ (resp. $X_{Nis}$,
$X_{et}$) of the open subsets $U$ of $X$ (resp. of $U\sr X$ etale),
with $F_X$ functorial in $X$: a map $f:X\sr Y$ induces $f^*:f^*F_Y\sr
F_X$.
\paragraph{1.2} The definition of the motivic cohomology groups of $X$
smooth over $k$ has the following form.

\noindent
{\bf a.} One defines for each $q\in\zz$ a complex of presheaves of
abelian groups $\zz(q)$ on $Sm/k$. It is in fact a complex of sheaves
for the etale topology, hence a fortiori for the Nisnevich and Zariski
topology. For any abelian group $A$ the same applies to $A(q):=A\oo
\zz(q)$. 

\noindent
{\bf b.} The motivic cohomology groups of $X$ with coefficients in $A$
are the hypercohomology groups of the $A(q)$, in the Nisnevich
topology:
$$H^{p,q}(X,A):={\bf H}^p(X_{Nis},A(q))$$
For $A=\zz$ we will write simply $H^{p,q}(X)$.

Motivic cohomology has the following properties:
\begin{enumerate}
\item the complex $\zz(q)$ is zero for $q<0$. For any $q$ it lives in
cohomological degree $\le q$. As a complex of Nisnevich sheaves it is
quasi-isomorphic to $\zz$ for $q=0$ and to ${\bf G}_m[-1]$ for $q=1$
\item $H^{p,p}(Spec(k))=K_p^M(k)$ for any $p\ge 0$
\item for any $X$ in $Sm/k$ one has
$$H^{p,q}(X)=CH^q(X,2q-p)$$
where $CH^q(X,2q-p)$ is the $(2q-p)$-th higher Chow group of cycles of
codimension $q$
\item in the etale topology, for $n$ prime to the characteristic of
$k$, the complex $\zz/n(q)$ is quasi-isomorphic to $\mu_{n}^{\oo q}$,
giving for the etale analog of $H^{p,q}$ the formula
$$H^{p,q}_{et}(X,\zz/n):={\bf
H}^{p}(X_{et},\zz/n(q))=H^p(X_{et},\mu_{n}^{\oo q})$$
\end{enumerate}
\paragraph{1.3} The category $SmCor(k)$ is the category with objects
separated schemes smooth of finite type over $k$, for which a morphism
$Z:X\sr Y$ is a cycle $Z=\sum n_i Z_i$ on $X\times Y$ each of whose
irreducible components $Z_i$ is finite over $X$ and projects onto a
connected component of $X$. A morphism $Z$ can be thought of as a
finitely valued map from $X$ to $Y$. For $x\in X$, with residue field
$k(x)$, it defines a zero cycle $Z(x)$ on $Y_{k(x)}$, and the
assumption made on $Z$ implies that the degree of this 0-cycle is
locally constant on $X$.

A morphism of schemes $f:X\sr Y$ defines a morphism in $SmCor(k)$: the
graph of $f$. This graph construction defines a faithful functor from
$Sm/k$ to the additive category $SmCor(k)$.

A {\em presheaf with transfers} is a contravariant additive functor
from the category $SmCor(k)$ to the category of abelian groups. The
embedding of $Sm/k$ in $SmCor(k)$ allows us to view a presheaf with
transfers as a presheaf on $Sm/k$ endowed with an extra structure. A
sheaf with transfers (for a given topology on $Sm/k$, usually the
Nisnevich topology) is a presheaf with transfers which, as a presheaf
on $Sm/k$, is a sheaf. The Nisnevich and the etale topologies have the
virtue that if $F$ is a presheaf with transfers, the associated sheaf
$a(F)$ carries a structure of a sheaf with transfers. This structure
is uniquely determined by $F\sr a(F)$ being a morphism of presheaves
with transfers. For any sheaf with transfers $G$, one has
$$Hom(a(F),G)\ii Hom(F,G)$$
(Hom of presheaves with transfers). All of this fails for the Zariski
topology.

The complexes $\zz(q)$ (or $A(q)$) start life as complexes of sheaves
with transfers. 

\paragraph{1.4} A presheaf $F$ on $Sm/k$ is called {\em homotopy invariant} if
$F(X)=F(X\times\af)$. As the point $0$ of $\af$ defines a section of
the projection of $X\times\af$ to $X$, for any presheaf of abelian
groups $F$, $F(X)$ is naturally a direct factor of $F(X\times\af)$; it
follows that the condition ``homotopy invariant'' is stable by kernels,
cokernels and extensions of presheaves. The following construction is
a derived version of the left adjoint to the inclusion
$$(homotopy\,\,\,\,\,invariant\,\,\,\,\, presheaves)\subset
(all\,\,\,\,\, presheaves)$$
\begin{description}
\item[a.]
For $S$ a finite set, let $A(S)$ be the affine space freely spanned
(in the sense of barycentric calculus) by $S$. Over $\bf C$, or $\bf
R$, $A(S)$ contains the standard topological simplex spanned by
$S$. The schemes $\Delta^m:=A(\{0,\dots,n\})$ form a cosimplicial
scheme. 
\item[b.] For $F$ a presheaf, $C_{\bullet}(F)$ (the ``singular complex of
$F$'') is the simplicial presheaf $C_n(F):X\mapsto F(X\times\Delta^{n})$.
\end{description}
Arguments imitated from topology show that for $F$ a presheaf of
abelian groups, the cohomology presheaves of the complex $C_{*}(F)$,
obtained from $C_{\bullet}(F)$ by taking alternating sum of the face
maps, are homotopy invariant. If $F$ has transfers so do the $C_n(F)$
and hence the $H_{n}C_*(F)$. A basic theorem proved last year is:
\begin{theorem}
\llabel{basic}
Let $F$ be a homotopy invariant presheaf with transfers over a perfect
field with the associated Nisnevich sheaf $a_{Nis}(F)$. Then the
presheaves with transfers 
$$X\mapsto H^i(X_{Nis},a_{Nis}(F))$$
are homotopy invariant as well.
\end{theorem}
The particular case of this theorem for $i=0$ claims the homotopy
invariance of the sheaf with transfers $a_{Nis}(F)$.

\vskip 5mm

Last year, the equivalence of a number of definitions of $\zz(q)$ was
proven. Equivalence means: a construction of an isomorphism in a
suitable derived category, implying an isomorphism for the
corresponding motivic cohomology groups. For our present purpose the
most convenient definition is as follows. 

Let $\zz_{tr}(X)$ be the sheaf with transfers represented by $X$ (on
the category $SmCor(k)$). We set 
$$
K_q=\left\{
\begin{array}{ll}
0 & \mbox{\rm for $q<0$}\\
\zz_{tr}({\bf
A}^q)/\zz_{tr}({\bf A}^{q}-\{0\})& \mbox{\rm for $q\ge 0$}
\end{array}
\right.
$$ 
and $\zz(q)=C_*(K_q)[-q]$.

\subsection{Motivic homotopy category}\label{3}
The motivic homotopy category $Ho_{\af,\bullet}(S)$ (pointed
$\af$-homotopy category of $S$), for $S$ a finite dimensional
noetherian scheme, will be the category deduced from a category of
simplicial sheaves by two successive localizations\footnote{In the
Appendix we have assembled the properties of ``localization'' to be
used in this talk and in the next}. 

One starts with the category $Sm/S$ of schemes smooth over $S$, with
the Nisnevich topology, and the category of pointed simplicial sheaves
on $Sm/S$. For any site $\cal S$ (for instance $(Sm/S)_{Nis}$), there
is a notion of {\em local  equivalence} of (pointed) simplicial
sheaves. It proceeds as follows.
\begin{description}
\item[a.] A sheaf $G$ defines a simplicial sheaf $G_*$ with all
          $G_n=G$ and all simplicial maps the identities. The functor
          $G\mapsto G_*$ has a left adjoint $F\mapsto \pi_0(F)$:
$$Hom(F_{*},G_{*})=Hom(\pi_0(F_*),G)$$
The sheaf $\pi_0(F_*)$ can be described as the equalizer of  $F_1\dsr
F_0$, as well as the sheaf associated to the presheaf
$$U\mapsto \pi_0(|F_{*}(U)|)$$
The same holds in the pointed context. We will often write simply $G$
for $G_*$.
\item[b.] If $F_*$ is a simplicial sheaf, and $u$ a section of $F_0$
over $U$, one also disposes of sheaves $\pi_i(F_*,u)$ over $U$: the
sheaves associated to the presheaves
$$V/U\mapsto \pi(|F_{*}(V)|,u)$$
\item[c.] A morphism $F_*\sr G_*$ is a local  equivalence, if it induces
an isomorphism on $\pi_0$ as well as, for any local section $u$ of
$F_0$, an isomorphism on all $\pi_i$. This applies also to pointed
simplicial sheaves: one just forgets the marked point.
\end{description}
One defines $Ho_{\bullet}(Sm/S)$ as the category derived from the
category of pointed simplicial sheaves on $(Sm/S)_{Nis}$ by formally inverting
local  equivalences. Until made more concrete, this definition
could lead to set-theoretic difficulties, which we leave the reader to
solve in its preferred way. 

For $G$ a pointed sheaf on $Sm/S$,  \ref{1.7} applies to
$G_*$ and to the localization by local equivalences: one has
\begin{equation}
\llabel{1.3}
Hom_{Ho_{\bullet}}(F_*,G_*)=Hom(F_*,G_*)=Hom(\pi_0(F_*),G)
\end{equation}
\begin{definition}
\llabel{d1.1}
An object $X$ of $Ho_{\bullet}(Sm/S)$ is called $\af$-local if for any
simplicial sheaf $Y$, one has
$$Hom_{Ho_{\bullet}}(Y,X)\ii Hom_{Ho_{\bullet}}(Y\times\af/*\times
\af,X)$$
At the right hand side, $/*\times \af$ means that in the product,
$*\times\af$ is contracted to a point, the new base point.
\end{definition} 
\begin{proposition}
\llabel{1.2}
For $G$ a pointed sheaf on $Sm/S$, the simplicial sheaf $G_{*}$ is
$\af$-local if and only if $G$ is homotopy invariant.
\end{proposition}
\begin{proof}
We have $\pi_0(Y\times \af)=\pi_0(Y)\times \af$, so that by
(\ref{1.3}) ``$\af$-local'' means that for any pointed sheaf $Y$, one
has
$$Hom(Y,G)=Hom(Y\times\af/*\times\af,G)$$
A morphism $Y\sr G$ can be viewed as the data, for each $y\in Y(U)$,
of $f(y)\in G(U)$, functorial in $U$ and marked point going to marked
point. A morphism $g:Y\times\af\sr G$ can similarly be described as
data for $y\in Y(U)$ of $g(y)\in G(U\times\af)$. Homotopy invariance
hence implies $\af$-locality. The converse is checked by taking for
$Y$ the disjoint sum of a representable sheaf and the base point.
\end{proof}
\begin{definition}
\llabel{d1.4}
{\bf (i)} A morphism $f:Y_1\sr Y_2$ in $Ho_{\bullet}(Sm/S)$ is an
$\af$-equivalence if for any $\af$-local $X$, one has in
$Ho_{\bullet}(Sm/S)$
$$Hom(Y_2,X)\ii  Hom(Y_1,X)$$ 
{\bf (ii)} The category $Ho_{\af,\bullet}(Sm/S)$ is deduced from
$Ho_{\bullet}(Sm/S)$ by formally inverting $\af$-equivalences.
\end{definition}
\begin{remark}\rm
\llabel{rm3.4} If a morphism in $Ho_{\bullet}(Sm/S)$ becomes an
isomorphism in $Ho_{\bullet,\af}(Sm/S)$ it is an
$\af$-equivalence. Indeed, if $X$ in $Ho_{\bullet}(Sm/S)$ is
$\af$-local, an application of \ref{1.7} shows that for any $Y$,
$$Hom_{Ho_{\bullet}}(Y,X)\sr Hom_{Ho_{\bullet,\af}}(Y,X)$$
is bijective. If $f:Y_1\sr Y_2$ in $Ho_{\bullet}(Sm/S)$ has an image
in $Ho_{\bullet,\af}(Sm/S)$ which is an isomorphism, it follows that
for any $\af$-local $X$, one has 
$$Hom_{Ho_{\bullet}}(Y_2,X)\ii Hom_{Ho_{\bullet}}(Y_1,X).$$
Such an $f$ is an $\af$-equivalence.
\end{remark}
\begin{example}\rm
\llabel{1.5}
Arguments similar to those given before show that if $G$ is a homotopy
invariant pointed sheaf, then for any simplicial pointed sheaf $F_*$,
one has
$$Hom_{Ho_{\af,\bullet}}(Sm/S)(F_*,G)=Hom(F_*,G)=Hom(\pi_0(F_*),G)$$
in particular, if $U$ is smooth over $S$ and if $U_+$ is the disjoint
union of $U$ and of a base point,
$$Hom_{Ho_{\af,\bullet}}(U_+,G)=G(U)$$
\end{example}

\subsection{Derived categories versus homotopy categories}\label{5}
For any topos $T$, which for us will be the category of sheaves on
some site $S$, the pointed homotopy category $Ho_{\bullet}(S)$ as well
as the derived category $D(S)$ are obtained by localization. For the
derived category, one starts with the category of complexes of abelian
sheaves. The subcategory of complexes living in homological degree
$\ge 0$ is naturally equivalent, by the Dold Puppe construction, to
the category of simplicial sheaves of abelain groups. The equivalence
is
$$N:(simplicial\,\, F_*)\mapsto complex(\bigcap_{i\ne 0}
ker(\partial_i), \partial_0)$$ 
We will write $K$ for the inverse equivalence. For $S$ a point, and
$\pi$ an abelian group, $|K(\pi[n])|$ is indeed the Eilenberg-Maclane
space $K(\pi,n)$. For a complex $C$ not assumed to live in homological
degree $\ge 0$, we define
$$K(C):= K(\tau_{\ge 0} C)$$
where $\tau_{\ge n}C$ is the subcomplex in $C$ of the form
$$\dots C_{n+2}\stackrel{d_{n+1}}{\sr} C_{n+1}\sr ker(d_{n})\sr 0$$
Note that $C\mapsto \tau_{\ge 0}C$ is right adjoint to the inclusion
functor 
$$
\left(
\mbox{\rm complexes in homological degree $\ge 0$}
\right)
\hookrightarrow
\left(
\mbox{\rm all complexes}
\right)
$$
so that $(N,K)$ form a pair of adjoint functors:
$$
\left(
\mbox{\rm simplicial abelian sheaves}
\right)
\rightleftarrows
\left(
\mbox{\rm complexes of sheaves}
\right)
$$
\begin{theorem}
\llabel{1.6}
Assume that $S=Spec(k)$ with $k$ perfect. Then, for $F$ a presheaf
with transfers, and $U_+$ as above, and $p\ge 0$
$$Hom_{Ho_{\af,\bullet}}(U_+,K(F[p]))={\bf H}^p(U_{Nis},C_*(F))$$
\end{theorem}
In this theorem $K(F[p])$ is the simplicial sheaf of abelian groups
whose normalized chain complex is $F$ in homological degree $p$.

To prove the theorem we establish the chain of equalities,
\begin{equation}\llabel{chain}
\begin{array}{cc}
{\bf
H}^p(U_{Nis},C_*(F))=Hom_{Ho_{\bullet}}(U_+,K(C_*(F)[p]))=\\\\
=Hom_{Ho_{\bullet,\af}}(U_+,
K(C_*(F)[p]))=Hom_{Ho_{\bullet,\af}}(U_+, K(F[p])),
\end{array}
\end{equation}
the first equality is proved right before  \ref{3p1}, the
second right after  \ref{3p2} and the last one follows
from  \ref{6.2}.

Let $forget$ be the forgetting functor from abelian sheaves to sheaves
of sets. Its left adjoint is $F\mapsto \zz[F]$: the sheaf associated
to the presheaf 
$$U\mapsto (\mbox{\rm abelian group freely generated by $F(U)$})$$
In the pointed context , the adjoint is
$$(F,*)\mapsto \tilde{\zz}(F):\zz(F)/\zz(*)$$
We have the same adjunction for (pointed) simplicial objects.
\begin{proposition}
\llabel{2.1}
On a site with enough points (and presumably always) , one has

\noindent
{\bf (i)} The functor $F_*\mapsto N\zz(F_*)$ from pointed simplicial
sheaves to complexes of abelian groups transforms local 
equivalences into quasi-isomorphisms

\noindent
{\bf (ii)} The right adjoint $C\mapsto forget(K(C))$ transforms
quasi-isomorphisms to local equivalences.
\end{proposition}
The assumption ``enough points'' applies to $Sm/k$ with the Nisnevich
topology: for any $U$ in $Sm/k$ and any point $x$ of $U$, $F\mapsto
F(Spec({\cal O}^h_{U,x}))$ is a point, and they form a conservative
system.  

\vskip 2mm
\noindent
\begin{proof}
Local equivalence (resp. quasi-isomorphism) can be checked point by
point, and the two functors considered commute with passage to the
fiber at a point. This reduces our proposition to the case when $S$ is
just a point, i.e. to usual homotopy theory. In that case, (i) boils
down to the fact that a weak equivalence induces an isomorphism on
reduced homology, a theorem of Whitehead, and (ii) reduces to the
fact: for a complex $C$, $\pi_i(K(C))$, computed using any base point,
is $H_i(C)$. The $\pi_i(K(C))$ are easy to compute because $K(C)$ has
the Kan property.
\end{proof}
Applying  \ref{1.8}, we deduce from  \ref{2.1}
the following.
\begin{proposition}
\llabel{2.2}
Under the same assumptions as in \ref{2.1}, for $F_{*}$ a pointed
simplicial sheaf and $C$ a complex of abelian sheaves, one has
\begin{equation}
\label{2.3}
Hom_{Ho_{\bullet}}(F_{*}, K(C))=Hom_{D}(N\tilde{\zz}(F_*), C)
\end{equation}
\end{proposition}
Let $F$ be a sheaf and $F_+$ be deduced from $F$ by adjunction of a
base point. We also write $F$ and $F_+$ for the corresponding
``constant'' simplicial sheaf. One has
$$N\tilde{\zz}(F_+)=N\zz(F)=(\zz(F)\,\,\,\mbox{\rm in degree zero})$$
For the pointed simplicial sheaf $F_+$, the group $Hom_D(\zz(F),C)$
which now occurs at the right hand side of (\ref{2.3}) can be
interpreted as hypercohomology of $C$ ``over $F$ viewed as a space'',
i.e. in the topos of sheaves over $F$. For $F$ defined by an object
$U$ of the site $S$, this is the same as hypercohomology of the site
$S/U$. As we do not want to assume $C$ bounded below (in cohomological
numbering), checking this requires a little care. 

For a complex of sheaves $K$ over a site $S$, not necessarily bounded
below, ${\bf H}^0(S,K)$ can be defined as the Hom group in the derived
category $Hom_{D}(\zz,K)$. For $F$ in a topos $T$ and the topos $T/F$:
``F viewed as a space'', besides the morphism of toposes
$(T/F)\stackrel{j}{\sr} T$, i.e. the adjoint pair $(j^*,j_*)$, we have
for abelian sheaves an adjoint pair $(j_!,j^*)$, with $j_!$ and $j^*$
both exact. By  \ref{1.8}, $(j_!,j^*)$ induce an adjoint
pair for the corresponding derived categories. As $j_!\zz=\zz[F]$, we
get
\begin{equation}
\llabel{2.4}
Hom_D(\zz(F),C)={\bf H}^0(T/F, j^*(C))
\end{equation} 
hence
\begin{equation}
\llabel{2.5}
Hom_{Ho_{\bullet}}(F_+,K(C))={\bf H}^0(T/F, j^*(C))
\end{equation} 

Let us consider the particular case of $Sm/k$ with the Nisnevich
topology. For any complex of sheaves, (\ref{2.5}) gives for $U$
smooth over $k$
\begin{equation}
\llabel{2.6}
Hom_{Ho_{\bullet}}(U_+,K(C))={\bf H}^0(U_{big-Nis}, C)
\end{equation}  
Here, $U_{big-Nis}$ is the site $(Sm/S)/U$ with the Nisnevich
topology. It has however the same hypercohomology as the small
Nisnevich site $U_{Nis}$. Indeed, one has a morphism
$\epsilon:U_{big,Nis}\sr U_{Nis}$ and the functors $\epsilon^*$ and
$\epsilon_*$ are exact. One again applies  \ref{1.8}. If
we apply (\ref{2.6}) to a translate (shift) of $C$, we get
\begin{equation}
\llabel{2.7}
Hom_{Ho_{\bullet}}(U_+,K(C[p]))={\bf H}^p(U_{Nis}, C)
\end{equation}  
Applying (\ref{2.7}) to $C_*(F)$ we get the first equality in
(\ref{chain}).
\begin{proposition}
\llabel{3p1}
Let $C$ be a complex of abelian sheaves on $Sm/k$. The following
conditions are equivalent:
\begin{enumerate}
\item $K(C)$ is $\af$-local
\item for $i\le 0$, the functor $U\mapsto {\bf H}^i(U,C)$ is homotopy
invariant
\item for any complex $L$ in cohomological degree $\le 0$, one has in
the derived category
$$Hom(L\oo\zz(\af), C)\ii Hom(L,C)$$
\end{enumerate}
\end{proposition}
\begin{proof}
By  \ref{2.2}, condition (1) can be rewritten:
for any pointed simplicial sheaf $F_{*}$.
$$Hom_D(N\tilde{\zz}(F_*),
C)=Hom_D(N\tilde\zz(F_*\times\af/*\times\af))$$
The operation $F_*\mapsto F_*\times\af/*\times\af$ is better written
as a smash product $F_*\wedge \af_+$ with $\af_+$. For pointed sets
$E$ and $F$, $\tilde{\zz}(E\wedge
F)=\tilde{\zz}(E)\oo\tilde{\zz}(F)$. It follows that 
$$\tilde{\zz}(F_*\times\af/*\times\af)=\tilde{\zz}(F_*\wedge\af_+)=\tilde{\zz}(F_*)\oo\tilde{\zz}(\af_+)=\tilde{\zz}(F_*)\oo\zz(\af)$$
(isomorphisms of simplicial sheaves), hence 
$$N\tilde{\zz}(F_*\times\af/*\times\af)=N\tilde{\zz}(F_*)\oo\zz(\af)$$
It follows that (1) is the particular case of (3) for $L$ of the form
$\tilde{\zz}(F_*)$. Similarly, (2) is the particular case of (3) for
$L$ of the form $\zz(U)[i]$, with $i\ge 0$. 

The suspension $\Sigma^iF_*$ of a simplicial pointed sheaf $F_*$ is
its smash product with the simplicial sphere $S^i$ (the $i$-simplex
modulo its boundary). It follows that
$$\tilde{\zz}(\Sigma^i F_*)=\tilde{\zz}(F_*)\oo \tilde{\zz}(S^i)$$
(isomorphism of simplicial sheaves), and by Eilenberg-Zilber, the
normalized complex $N\zz(\Sigma^iF_*)$ is homotopic to the tensor
product of the normalized complexes of $\tilde{\zz}(F_*)$ and
$\tilde{\zz}(S^i)$. The latter is simply $\zz[i]$:
$$N\tilde{\zz}(\Sigma^i F_*)\cong \tilde{\zz}(F_*)[i]$$
This is just a high-brown way of telling that the reduced homology of
a suspension is just a shift of the reduced homology of the space one
started with. 

Applying this to $F_*=U_+$, one obtains that
(1)$\Rightarrow$(3). Indeed, $\tilde{\zz}(\Sigma^i U_+)$ is homotopic
to $\tilde{\zz}(U_+)[i]=\zz(U)[i]$.

We now prove that (2)$\Rightarrow$(1). For $L$ a complex, let (*) be
the statement that the conclusion of (3) holds for all $L[i]$, $i\ge
0$. The assumption (2) is that (*) holds  for $L$ reduced to $\zz(U)$
in degree $0$, and we will conclude that it holds for all $L$ in
cohomological degree $\le 0$ by ``devissage'':

\vskip 2mm
\noindent
{\bf (a)} The case of a sum of $\zz(U)$, in degree zero, follows from
Corollary \ref{c3.4}.

\vskip 2mm
\noindent
{\bf (b)} Suppose that $L$ is bounded, is in cohomological degree $\le 0$
and that (*) holds for all $L^n$. The functors 
$$h':L\mapsto Hom^n(L,C)$$
$$h'':L\mapsto Hom^n(L\oo\zz(\af), C)$$
are contravariant cohomological functors, hence give rise to
convergent spectral sequences
$$E^{pq}_1=h^q(L^{-p})\Rightarrow h^{p+q}(L).$$
One has a morphism of spectral sequences
$$E(\mbox{for } h')\sr E(\mbox{for } h'')$$
which is an isomorphism for $q\le 0$, and both $E^{pq}$ vanish for
$p<0$ or $p$ large. It follows that $h^{'n}(L)\ii h^{''n}(L)$ for
$n\le 0$, i.e. that $L$ satisfies (*).

The same argumnet can be expresed as an induction on the number of $i$
such that $L^i\ne 0$. If $n$ is the largest (with $n\le 0$), the
induction assumption applies to $\sigma^{<n}L$, even to $(\sigma^{<n}
L)[-1]$, and one concludes by the long exact sequence defined by
$$0\sr L^{n}[-n]\sr L\sr \sigma^{<n}L\sr 0$$

\vskip 2mm
\noindent
{\bf (c)} Expressing $L$ as the inductive limit of the $\sigma^{\ge -n}L$
and using  \ref{3.5}, one sees that we need not assume
that $L$ is bounded.

\vskip 2mm
\noindent
{\bf (d)} If $L'\sr L''$ is a quasi-isomorphism, $L'\oo\zz(\af)\sr
L''\oo\zz(\af)$ is one too (flatness of $\zz(\af)$), and (*) holds for
$L'$ if and only if it holds for $L''$. 

\vskip 2mm
\noindent
{\bf (e)} Any abelian sheaf $L$ is a quotient of a direct sum of
sheaves $\zz(U)$. For instance, the sum over $(U,s)$, $s\in
\Gamma(U,L)$, of $\zz(U)$ mapping to $L$ by $s$. It follows that $L$
admits a resolution $L^*$ by such sheaves. By (a) and (c), $L^*$
satisfies (*). It follows from (d) that $L$ satisfies (*) and then by
(c) that any complex in degree $\le 0$ satisfies (*).
\end{proof}

\subsection{Application to presheaves with transfers}\label{5.5}
Let $F$ be a presheaf with transfers. A formal argument (\cite{CC})
shows that the presheaves with transfers $H^iC_*(F)$ are homotopy
invariant. By the basic result ( \ref{basic}) recalled in the
first lecture, it follows that for any $U$, one has
\begin{equation}
\llabel{2.8}
{\bf H}^*(U,C_*(F))={\bf H}^*(U\times\af, C_*(F))
\end{equation}  
Indeed, as $U$ and $U\times\af$ are of finite cohomological dimension,
both sides are  abutment of a convergent spectral sequence
$$E^{pq}_2=H^p(U,H^qC_*(F))\Rightarrow {\bf H}^{p+q}(U,C_*(F))$$
and the same for $U\times\af$. By  \ref{basic} applied to
$H^qC_*(F)$, 
$$H^p(X, H^qC_*(F)):=H^p(X, aH^qC_*(F))$$
is the same for $X=U$ or for $X=U\times\af$. Applying (\ref{2.7}), we
conclude from  \ref{3p1}((2)$\Rightarrow$(1)) that
\begin{proposition}
\llabel{3p2}
For $k$ perfect, if $F$ is a presheaf with transfers, for all $p$,
$K(C_*(F)[p])$ is $\af$-local.
\end{proposition}
Combining  \ref{3p2} with  \ref{1.7} we get
the second equality in (\ref{chain}).

\subsection{End of the proof of  \ref{1.6}\label{7}}
For any pointed simplicial sheaf $G_{\bullet}$, $C_{\bullet}(G_{\bullet})$ is
a pointed bisimplicial sheaf of which one can take the diagonal
$\Delta C_{\bullet}(G_{\bullet})$. For any pointed sheaf $G$, one has
a natural map $G\sr C\BB(G)$, and for a pointed simplicial sheaf
$G_{\bullet}$, those maps for the $G_n$ induce
$$a: G_{\bullet}\sr \Delta C_{\bullet}(G_{\bullet})$$
\begin{proposition}
\llabel{6.3}
The morphism $a:G_{\bullet}\sr \Delta C_{\bullet}(G_{\bullet})$ is an
$\af$-equivalence.
\end{proposition}
\begin{proof}
We deduce  \ref{6.3} from \ref{13.6}.

The two maps $0,1:F_{\bullet}\sr F_{\bullet}\wedge \af_{+}$ are
equalized by $F_{\bullet}\wedge \af_+\sr F_{\bullet}$, hence become
equal in the $\af$-homotopy category. If two maps of pointed
simplicial sheaves $F\BB\dsr G\BB$ factor as $F\BB\dsr
F\BB\wedge\af_+\sr G\BB$, they also become equal. By the adjunction of
$\wedge\af_+$ and of $C_1(-)=\uu{Hom}(\af_+,-)$, such a factorization
can be rewritten as 
$$F\BB\sr C_1(G\BB)\sr G\BB$$
Particular case: the maps $C_1(G\BB)\sr G\BB$, become equal in the
homotopy category. Evaluated on $X$, these maps are the restriction
maps $0^*,1^*:G\BB(X\times\af)\sr G\BB(X)$.

The affine space ${\bf A}^n$ is homotopic to a point in the sense that
$H:\af\times{\bf A}^n\sr {\bf A}^n:(t,x)\mapsto tx$ interpolates
between the identity map (for $t=1$) and the constant map $0$ (for
$t=0$). The map $H$ induces
$$G\BB(S\times{\bf A}^n)\sr G\BB(S\times{\bf
A}^n\times\af)$$
and, composing with $0$, $1$ in $\af$, we obtain that
$$G\BB(S\times{\bf A}^n)\dsr  G\BB(S\times{\bf A}^n)$$
the identity map, and the map induced by $0:{\bf A}^n\sr {\bf A}^n$,
are equal in the $\af$-homotopy category. The map of simplicial
sheaves $G\BB\sr C_nG\BB$ is hence an $\af$-equivalence. It has
as inverse in the $\af$-homotopy category the map induced by
$0:Spec(k)\sr {\bf A}^n$ and one applies  \ref{rm3.4}. We now
apply  \ref{13.6} to the bisimplicial sheaves 
$$G_{pq}:=G_p$$
$$H_{pq}:=C_q G_p: S\mapsto G_p(S\times \Delta^q)$$
and to the natural map $G_{pq}\sr H_{pq}$. For fixed $q$, this is just
$G(S)\sr G(S\times{\bf A}^q)$, and  \ref{13.6} gives
 \ref{6.3}.
\end{proof}
To prove the last equality in (\ref{chain}), it suffices to show that:
\begin{lemma}
\llabel{6.2} For any abelian sheaf $F$, $F[p]\sr C_*(F)[p]$ induces an
$\af$-equivalence from $K(F[p])$ to $K(C_*(F)[p])$.
\end{lemma}
\begin{proof} For $G$ a monoid (with unit), the pointed simplicial
set $B_{\bullet}G$ is given by 
$$
B_nG = \left\{
\begin{array}{c}
\mbox{\rm functors from the ordered set $(0,\dots,n)$ viewed as a category}\\
\mbox{\rm to $G$
viewed as a category with one object}
\end{array}
\right\}
$$
This construction can be sheafified, and can be applied termwise to a
simplicial sheaf of monoids, leading to a pointed bisimplicial sheaf
of which one can take the diagonal
$$BG_{\bullet} := \Delta B_{\bullet}(G_{\bullet})$$
This construction commutes with the construction $G_{\bullet}\sr
\Delta C_{\bullet}(G_{\bullet})$. Indeed,
$B_nG_p$ is naturally isomorphic to
$G_p^n$, the operation $C_m$ commutes with products, and
$B(\Delta C_{\bullet}(G_{\bullet}))$ and $\Delta C_{\bullet}(B
G_{\bullet})$ are both diagonals of the trisimplicial pointed sheaf
$C_{\bullet}B_{\bullet}G_{\bullet}$.

For abelian simplicial sheaves, the operation $B$ gives again abelian
simplicial sheaves, hence can be iterated, and $\Delta C_{\bullet}$
commutes with $B^n$.

Via Dold-Puppe construction, $B$ corresponds, up to homotopy, to the
shift $[1]$ of complexes:
$$N B G_{\bullet}\cong (N G_{\bullet})[1].$$
This can be viewed as an application of the Eilenberg-Zilber Theorem
(see \cite[Th. 8.5.1]{Chuck}): one has 
$$NBG\BB \cong BG_*\cong Tot B_*G_* \,\,\,\,\,\,\,\mbox{\rm
(Eilenberg-Zilber),}$$
and for each $G_q$, the normalization of $B\BB G_q$ is just $G_q[1]$,
so that the double complexes $B_*G_*$ and $H_{pq}:=G_q$ for $p=1$, $0$
otherwise, have homotopic $Tot$.

If $G_{\bullet}$ is an abelian simplicial sheaf, applying 
\ref{6.3} to $B^pG_{\bullet}$, we obtain that 
\begin{equation}
\llabel{new}
B^p G_{\bullet}\sr \Delta C_{\bullet} B^p G_{\bullet}=B^p \Delta
C_{\bullet} G_{\bullet}
\end{equation}
is an $\af$-equivalence. The functor $K$ transforms chain
homotopy equivalences into simplicial equivalences. For any simplicial
abelian group $L\BB$ (to be $G\BB$ or $\Delta C\BB G\BB$), we hence
have a simplicial homotopy equivalence
$$B^pL\BB = KN B^p L\BB \cong K((NL\BB)[p])$$
Simplicial homotopy equivalences being $\af$-equivalences, we
conclude that (\ref{new}) induces an $\af$-equivalence
$$K((NG_{\bullet})[p])\sr K(N(\Delta C_{\bullet}
G_{\bullet})[p])$$ 
\end{proof}

\subsection{Appendix. Localization}\label{4}
Let $C$ be a category and $S$ be a set of morphisms of $C$. The
localizaed category $C[S^{-1}]$ is deduced from $C$ by ``formally
inverting all $s\in S$''. With this definition, it is clear that one
has a natural functor $loc:C\sr C[S^{-1}]$, bijective on the set of
objects, and that for any category $D$, 
$$F\mapsto F\circ loc: Hom(C[S^{-1}], D)\sr Hom(C,D)$$
is a bijection from $Hom(C[S^{-1}], D)$ to the set of functors from
$C$ to $D$ transforming morphisms in $S$ into isomorphisms. 

If one remembers that the categories form a 2-category, and if one
agree with the principle that one should not try to define a category
more precisely than up to equivalence (unique up to unique
isomorphism), the universal property of $C[S^{-1}]$ given above is
doubly unsatisfactory. The easily checked and useful universal
property is the following: $F\mapsto F\circ loc$ is an equivalence
from the category $Hom(C[S^{-1}], D)$ to the full subcategory of
$Hom(C,D)$ consisting of the functors $F$ which map $S$ to
isomorphisms. 
\begin{proposition}
\llabel{1.7}
If $Y$ in $C$ is such that the functor 
$$h_Y:C^{op}\sr Sets:X\mapsto Hom_{C}(X,Y)$$ 
transforms maps in $S$ into bijections , then
$$Hom_{C}(X,Y)\ii Hom_{C[S^{-1}]}(X,Y)$$
\end{proposition}
\begin{proof}
By Yoneda construction $Y\mapsto h_Y$, $C$ embeds
into the category $C^{\wedge}$ of contravariant functors from $C$ to
$Sets$, while $C[S^{-1}]$ embeds into $C[S^{-1}]^{\wedge}$, identified
by (a) with the full subcategory of $C^{\wedge}$ consisting of $F$
transforming $S$ into bijections. For $Y$ in $C$, with image $\bar{Y}$
in $C[S^{-1}]$, and for any $F$ in $C(S^{-1})^{\wedge}\subset
C^{\wedge}$, one has in $C^{\wedge}$
$$Hom(h_Y,F)=Hom(h_{\bar{Y}},F).$$
Indeed, by (a) and Yoneda lemma for $C$ and $C[S^{-1}]$ both sides are
$F(Y)$. This means that $h_{\bar{Y}}$ is the solution of the universal
problem of mapping $h_Y$ into an object of $C[S^{-1}]^{\wedge}\subset
C^{\wedge}$. In particular, for $Y$ as in (b), i.e. in
$C[S^{-1}]^{\wedge}$, $h_{\bar{Y}}$ coincides with $h_Y$, as claimed
by (b).
\end{proof}
\begin{proposition}
\llabel{1.8}
Let $(L,R)$ be a pair of adjoint functors bewteen categories $C$ and
$D$. Let $S$ and $T$ be sets of morphisms in $C$ and $D$. Assume that
$F$ maps $S$ to $T$ and that $G$ maps $T$ to $S$. Then the functors
$\bar{L}$, $\bar{R}$ bewteen $C[S^{-1}]$ and $D[T^{-1}]$ induced by $L$ and
$R$ again form an adjoint pair.
\end{proposition}
\begin{proof}
The functors $\bar{L}$ and $\bar{R}$ induced by $F$ and $G$ are
characterized by commutative diagrams
$$
\begin{CD}
C @>L>> D\\
@VVV @VVV\\
C[S^{-1}] @>\bar{L}>> D[T^{-1}]
\end{CD}
\,\,\,\,\,\,\,\,\,\,\,\,\,\,\,\,\,\,\,\,
\begin{CD}
D @>R>> C\\
@VVV @VVV\\
D[T^{-1}] @>\bar{R}>> C[S^{-1}]
\end{CD}
$$
Adjunction can be expressed by the data of $\epsilon:Id\sr RL$ and
$\eta:LR\sr Id$ such that the compositions
$$R\sr RLR\sr R$$
$$L\sr LRL\sr L$$
are the identity automorphisms of $R$ and $L$ respectively (see
e.g. \cite{McLane2}).

By the universal property of localization, $\epsilon$ induces a
morphism $\bar{\epsilon}\sr \bar{R}\bar{L}$, indeed, to define such a
morphism amounts to defining a morphism $loc\sr \bar{R}\bar{L}loc$,
and $\bar{R}\bar{L}loc=loc R L$. Similarly, $\eta$ induces
$\bar{\eta}:\bar{L}\bar{R}\sr Id$. The morphism $\bar{L}\sr
\bar{L}\bar{R}\bar{L} \sr \bar{L}$ is induced by $L\sr LRL\sr L$,
similarly for $\bar{R}\sr \bar{R}\bar{L}\bar{R}\sr \bar{R}$, and the
proposition follows.
\end{proof}
\begin{proposition}
\llabel{3.3}
Suppose that
\begin{enumerate}
\item the localization $C[S^{-1}]$ gives rise to a right calculus of
fractions
\item coproducts exist in $C$, and $S$ is stable by coproducts.
\end{enumerate}
Then, a coproduct in $C$ is also a coproduct in $C[S^{-1}]$.
\end{proposition}
For the definition of ``gives rise to a right calculus of fractions''
see \cite{}. It implies that for $X$ in $C$, the category of $s:X'\sr
X$ with $s$ in $S$ is filtering, and that 
$$Hom_{C[S^{-1}]}(X,Y)=colim_{s:X'\sr X} Hom_{C}(X',Y)$$
\begin{proof}
For $X$ in $C$, let $(S/X)$ be the filtering category of morphisms
$X'\sr X$ in $S$. For $X$ the coproduct of $X_{\alpha}$, $\alpha\in
A$, one has a functor ``coproduct'':
$$\prod (S/X_{\alpha})\sr (S/X)$$
It is cofinal: for $s:X'\sr X$ in $S$, one can construct a diagram
$$
\begin{CD}
X'_{\alpha} @>>> X_{\alpha}\\
@VVV @VVV\\
X' @>>> X
\end{CD}
$$
with $s_{\alpha}:X'_{\alpha}\sr X$ in $S$, and $\coprod s_{\alpha}$
dominates $s$. For any $Y$, it follows that 
$$Hom_{C[S^{-1}]}(X,Y)=colim_{(S/X)}
Hom_{C}(X',Y)=$$
$$=colim_{\prod(S/X_{\alpha})} Hom(\coprod X'_{\alpha}, Y)
= colim_{\prod(S/X_{\alpha})} \prod Hom(X'_{\alpha}, Y)=$$
$$=\prod colim_{S/X_{\alpha}} Hom(X_{\alpha}', Y)=\prod
Hom_{C[S^{-1}]}(X_{\alpha}, Y),$$
meaning that $X$ is also the coproduct of the $X_{\alpha}$ in
$C[S^{-1}]$.
\end{proof}
\begin{cor}
\llabel{c3.4}
Suppose that in the abelian category $A$ arbitrary direct sums exist
and are exact. Then, arbitrary direct sums exist in the derived
category $D(A)$, and the localization functor
$$C(A)\sr D(A)$$
commutes with direct sums.
\end{cor}
\begin{proof}
The functor $C(A)\sr D(A)$ factors through the category $K(A)$ of
complexes and maps up to homotopy. Direct sums in $C(A)$ are also
direct sums in $K(A)$. Indeed,
$$Hom_{K(A)}(\oplus K_{\alpha}, L)=H^0 Hom^{\bullet} (\oplus
K_{\alpha}, L)= $$
$$=H^{0}\prod Hom^{\bullet}(K_{\alpha}, L)=\prod H^0
Hom^{\bullet}(K_{\alpha},L),$$
as $\prod$ is exact for abelian groups. Exactness of $\oplus$ in $A$
ensures that a direct sum of quasi-isomorphisms is again a
quasi-isomorphism, and  \ref{3.3} applies to $K(A)$ and
the set $S$ of quasi-isomorphisms, proving the corollary.
\end{proof}
If $A_i$, $i\ge 0$ is an inductive system of objects of $A$, the
colimit of $A_i$ is the cokernel
$$\oplus A_i\stackrel{d}{\sr} \oplus A_i\sr colim\, A_i\sr 0$$
of the difference of the identity map and of the sum of the $A_i\sr
A_{i+1}$. If taking the inductive limit of a sequence is an exact
functor, the map $d$ is injective: it is the colimit of the 
$$\oplus_{i=0}^n A_i\sr \oplus_{i=0}^{n+1} A_i$$
each of which is injective, as its graded for the filtration by the
$\oplus_{i\ge p}A_i$ is the identity inclusion. 

Under the assumptions of Corollary \ref{c3.4} if a complex $K$ is the
colimit of an inductive sequence $K_{(i)}$, and if the sequence
\begin{equation}
\llabel{1}
0\sr \oplus K_{(i)}\stackrel{d}{\sr} \oplus K_{(i)}\sr K\sr 0
\end{equation}
is exact, then for any $L$, the long exact sequence of cohomology
reads
$$\sr Hom(K,L)\sr \prod Hom(K_{(i)},L)\stackrel{d}{\sr} \prod
Hom(K_{(i)},L)\sr$$
The kernel of $d$ is simply the projective limit of the
$Hom(K_{(i)},L)$. The cokernel is $lim^1$. One concludes.
\begin{proposition}
\llabel{3.5}
Suppose that in $A$ countable direct sums exist and are exact. If the
complex $K$ is the colimit of the $K_{(i)}$, and if the sequence
(\ref{1}) is exact, for instance if either
\begin{enumerate}
\item in $A$ inductive limits of sequences are exact
\item in each degree $n$, each $K_{(i)}^n\sr K_{(i+1)}^n$, is the
inclusion of a direct factor 
\end{enumerate}
then, one has a short exact sequence
$$0\sr lim^1 Hom(K_{(i)},L[-1])\sr Hom(K,L)\sr lim Hom(K_{(i)},L)\sr
0$$
\end{proposition}
\begin{proof}
It remains to check that condition (2) implies the exactness of
(\ref{1}). This is to be seen degree by degree. By assumption, the
$A_i:=K_{(i)}^n$, have decompositions compatible with the transition
maps $A_i=\oplus_{j=0}^i B_i$. A corresponding decomposition of
(\ref{1}) in direct sum follows, and we are reduced to check
exactness of the particular case (\ref{1})$_{(A,n)}$ of (\ref{1})
when $B_i=0$ for $i\ne n$, i.e. when $A_i$ is a fixed $A$ fro $i\ge
n$, and is $0$ otherwise. Let (\ref{1})$_0$ be the sequence
(\ref{1})$_{\zz,n}$ in $Ab$ for $A=\zz$. It is a split exact
sequence of free abelian groups. Because direct sums exist, $L\oo A$,
for $L$ a free abelian group is defined and functorial in $L$. It is a
sum of copies of $A$, indexed by a basis of $L$, and is characterized
by 
$$Hom(L\oo A, B) = Hom(L, Hom(A,B))$$
(functorial in $B$). The sequence (\ref{1})$_{A,n}$ is
(\ref{1})$_{0}\oo A$ and, (\ref{1})$_0$ being split exact, it
splits and in particular is exact.  
\end{proof}

The truncation $\sigma_{\le n}K=\sigma^{\ge -n}K$ of acomplex $K$ is
the subcomplex which coincides with $K$ in homological degree $\le n$
and is $0$ in homological degree $>n$.  For any complex $K$, one has
$$K=colim \sigma_{\le n} K$$
and this colimit satisfies condition (2) of  \ref{3.5}. It
follows that
\begin{cor}
\llabel{c3.6}
Under the assumptions of Corollary \ref{c3.4}, for any $K$ and $L$,
one has a short exact sequence
$$0\sr lim^1 Hom(\sigma_{\le n}K, L[-1])\sr Hom(K,L)\sr lim
Hom(\sigma_{\le n} K, L)\sr 0$$
\end{cor}

\section{$\af$-equivalences of simplicial sheaves on $G$-schemes}

\subsection{Sheaves on a site of $G$-schemes}\llabel{10}
We fix a base scheme $S$, supposed to be separated noetherian and of
finite dimension; fiber product $X\times_S Y$ will be written simply
as $X\times Y$. We also fix a group scheme $G$ over $S$, supposed to
be finite and flat. We note $h_X$ the representable sheaf defined by
$X$.

Let $QP/G$ be the category of schemes quasi-projective over $S$, given
with an action of $G$.  Any $X$ in $QP/G$ admits an open covering
$(U_i)$ by affine open subschemes which are $G$-stable. This makes it
possible to define a reasonable quotient $X/G$ in the category of
schemes over $S$ (rather than in the larger category of algebraic
spaces). For each $U_i$, $U_i/G$ is defined as the spectrum of the
equalizer 
$${\cal O}(U_i)\dsr {\cal O}(U_i\times G),$$
and $X/G$ is obtained by gluing the $U_i/G$. It is a categorical
quotient, i.e. the coequalizer of $G\times X\dsr X$. The map $X\sr
X/G$ is finite, open, and the topological space $|X/G|$ is the
coequalizer of the map of topological spaces $|G\times X|\dsr
|X|$. One can show that $X/G$ is again quasi-projective. Remark
\ref{rm9.2} below shows that this fact, while convenient for the
exposition, is irrelevant.

One defines on $QP/G$ a pretopology (\cite[II.1.3]{SGA3}) by taking as
coverings the family of etale maps $Y_i\sr X$ with the following
property: $X$ admits a filtration by closed equivariant subschemes
$\emptyset =X_n\subset\dots\subset X_1\subset X_0=X$ such that for
each $j$, some map $Y_i\sr X$ has a section over $X_j-X_{j+1}$. The
{\em Nisnevich topology} on $QP/G$ is the topology generated by this
pretopology. The category $QP/G$ with the Nisnevich topology is
the {\em Nisnevich site} $(QP/G)_{Nis}$.
\begin{remark}\rm
The corresponding topos is not the classifying topos of
\cite[IV.2.5]{SGA4}. A morphism $X\sr Y$ can become a Nisnevich
covering after forgetting the action of $G$, and not be a Nisnevich
covering. Example: $S=Spec(k)$, $G=\zz/2$, $X=S$, $Y=S\coprod S$ and
$G$ permutes two copies of $S$ in $Y$.
\end{remark}
\begin{remark}
\rm\llabel{rm9.2} Let $(affine/G)_{Nis}$ be the site defined as
above, with ``quasi-projective'' replaced by ``affine''. It is
equivalent to $(QP/G)_{Nis}$, in the sense that restriction to
$(affine/G)_{Nis}$ is an equivalence from the category of sheaves on
$(QP/G)_{Nis}$ to the category of sheaves on $(affine/G)_{Nis}$.
\end{remark}
\begin{remark}
\rm\llabel{rm9.3} If $G$ is the trivial group $e$, the definition
given above recovers the usual Nisnevich topology. For $G=e$, the
condition usually considered: ``every point $x$ of $X$ is the image of
a point with the same residue field of some $Y_i$'', is indeed
equivalent to the condition imposed above. This is checked by
noetherian induction: if a generic point $\xi$ of $X$ can be lifted to
$Y_i$, some open neighborhood $U\subset X$ of $\xi$ can be lifted to
$Y_i$, and one applies the induction hypothesis to
$X_1=(X-U)_{red}$. 
\end{remark}
We write $(QP)_{Nis}$ for the category of
quasi-projective schemes over $S$, with the Nisnevich topology. 
\begin{lemma}
\llabel{10.4}
If ${\cal U}:Y_i\sr X\,\,\,(i\in I)$ is a covering of $X$ in
$(QP/G)_{Nis}$, there is a covering $\cal V$ of $X/G$ in $(QP)_{Nis}$ whose
pull-back to $X$ is finer than $\cal U$.
\end{lemma}
\begin{proof}
Fix a filtration $\emptyset=X_n\subset\dots\subset X_0=X$ as in the
definition of the Nisnevich topology. We write $q$ for the quotient
map $X\sr X/G$. For $x$ in $X/G$, $(q^{-1}(x))_{red}$ is in some
$X_j-X_{j+1}$, by equivariance of the $X_j$, and one of the maps
$Y_i\sr X$ has an equivariant section $s$ over $X_j-X_{j+1}$. Let
$(X/G)_x^h$ be the henselization of $X/G$ at $x$. The map $q$ being
finite, the pull-back of $(X/G)_x^h$ to $X$ is the coproduct of the
$X_y^h$ for $q(y)=x$. The map from $Y_i$ to $X$ being etale, the
section $s$, restricted to $(q^{-1}(y))_{red}$, extends uniquely to a
section (automatically equivariant) of $Y_i$ over $\coprod_{q(y)=x}
X^h_y$. Writing $(X/G)^h_x$ as the limit of etale neighborhoods of
$x$, one finds that $x$ has an etale neighborhood $V(x)$ such that
$Y_i$ has an equivariant section over $X\times_{X/G} V(x)$. The $V(x)$
form the required covering $\cal V$.
\end{proof}

We define the $G$-local henselian schemes to be the schemes $Y$
obtained in the following way. For $X$ in $(QP/G)$, $y$ a point of
$X/G$, and $(X/G)^h_y$ the henselization of $X/G$ at $y$, take the
fiber product $Y:=X\times_{X/G} (X/G)^h_y$. As $X$ is finite over
$X/G$, this fiber product is a finite disjoint union of local
henselian schemes, and $G$-local henselian schemes are simply the
$G$-equivariant finite disjoint unions of $Y$ of local henselian
schemes, for which $Y/G$ is local. 
\begin{proposition}
\llabel{prop9.star}
If $Y$ is $G$-local henselian, the functor $X\mapsto Hom(Y,X)$ is a
point of the site $(QP/G)_{Nis}$, i.e. it defines a morphism of
the punctual site $(Sets)$ to $(QP/G)_{Nis}$. If $Y=X\times_{X/G}
(X/G)^h_{y}$, the corresponding fiber functor is $F\mapsto colim
F(X\times_{X/G} V)$, the colimit being taken over the etale
neighborhoods of $y$ in $X/G$. The collection of fiber functors so
obtained is conservative.
\end{proposition}
\begin{proof}
The functor $X\mapsto Hom(Y,X)$ commutes with finite limits. It
follows from  \ref{10.4} that it transforms coverings into
surjective families of maps, hence is a morphism of sites $(Sets)\sr
(QP/G)_{Nis}$.

To check that the resulting set of fiber functors is conservative, it
suffices to check that a family of etale $f_i:U_i\sr X$ is a covering
if for any $G$-local henselian $Y$,
$$\coprod Hom(Y,U_i)\sr Hom(Y,X)$$
is onto. The proof, parallel to that of  \ref{10.4} is left to the
reader.
\end{proof}

\subsection{The Brown-Gersten closed model structure on simplicial
sheaves on $G$-schemes}\llabel{12}
We recall that a commutative square of simplicial sets (or pointed
simplicial sets) 
\begin{equation}
\llabel{12.1.1}
\begin{CD}
K @>>> L\\
@VVV @VVV\\
M @>>> N
\end{CD}
\end{equation}
is {\em homotopy cartesian} (or a {\em homotopy pull-back square}) if,
when $L$ is replaced by $L'$ weakly equivalent to it and mapping to
$N$ by (Kan) fibration: $L\stackrel{\cong}{\sr}L'\sr N$, the map from
$K$ to $L'\times_N M$ is a weak equivalence.
\begin{definition}
\llabel{d12.5}
A simplicial presheaf $F_{\bullet}$ on $(QP/G)_{Nis}$ is {\em flasque}
if $F(\emptyset)$ is contractible and if for any (upper) distinguished
square:
$$
\begin{CD} 
B @>>> Y\\
@VVV @VVpV\\
A @>j>> X
\end{CD}
$$
($p$ - etale, $j$ open embedding, $B=p^{-1}(A)$ and $Y-B\cong X-A$),
the square
$$
\begin{CD} 
F(X) @>>> F(Y)\\
@VVV @VVpV\\
F(A) @>j>> F(B)
\end{CD}
$$
is homotopy cartesian.
\end{definition}
\begin{theorem}
\llabel{12.6} Let $f:F_{\bullet}\sr F_{\bullet}'$ be a morphism of
flasque simplicial\\ presheaves. If the induced morphism of simplicial
sheaves $aF_{\bullet}\sr aF_{\bullet}'$ is a local equivalence,
then, for any $U$ in 
$QP/G$, $F_{\bullet}(U)\sr F_{\bullet}'(U)$ is a weak equivalence.
\end{theorem}
\begin{proof}
For a $G$-scheme $X$ let $X_{Nis}$ be the small Nisnevich site of $X$
and for a presheaf $F$ on $(QP/G)$ let $F_{|X}$ be the restriction of
$F$ to $X_{Nis}$. Our assumption that $aF_{\bullet}\sr aF_{\bullet}'$
is a local equivalence implies that $aF_{\bullet,|U}\sr
aF_{\bullet,|U}'$ is a local equivalence. The map $U\sr U/G$
defines a morphism of sites $p:U_{Nis}\sr (U/G)_{Nis}$ and 
\ref{10.4} implies that the direct image functor $p_*$ commutes with
the associated sheaf functor and takes local equivalences to
local  equivalences. Therefore the morphism $a p_*(F_{\bullet,|U})\sr a
p_*(F_{\bullet,|U}')$ is a local  equivalence. The presheaves
$p_*(F_{\bullet,|U})$ 
and $p_*(F_{\bullet,|U}')$ are flasque on $(U/G)_{Nis}$ and by \cite[Lemma
3.1.18]{MoVo} we conclude that
$$F_{\bullet}(U)=p_*(F_{\bullet,|U})(U/G)\sr
p_*(F_{\bullet,|U}')(U/G)=F_{\bullet}'(U)$$
is a weak equivalence.
\end{proof}
In \cite{KSB2}, Brown and Gersten define a simplicial closed model
structure on the category of pointed simpicial sheaves on a Noetherian
topological space of finite dimension. As in Joyal \cite{}, the 
equivalences are the local equivalences . The homotopy category
is hence the same as Joyal's, but the model structure is different:
less cofibrations, more fibrations.

The arguments of \cite{KSB2} work as well in the Nisnevich topology,
for the big as well as for the small Nisnevich site, or for
$(QP/G)_{Nis}$, once  \ref{12.6} is available.

We review the basic definitions, working in $(QP/G)_{Nis}$. Let
$\Lambda^{n,k}$ be the sub-simplicial set of $\partial\Delta^n$, union
of all faces but the $k$-th face. For $n=0$,
$\Lambda^{0,0}=\emptyset$. One takes as {\em generating trivial
cofibrations} the maps of the form $(J)$:
\begin{description}
\item[$(J_a)$] $(\Lambda^{n,k}\times h_X)_+ \sr (\Delta^{n}\times
h_X)_+$
\item[$(J_b)$] for $U\sr X$ an open embedding,
$$(\Delta^n\times h_U\coprod_{\Lambda^{n,k}\times
h_U}\Lambda^{n,k}\times h_X)_+\sr (\Delta^{n}\times
h_X)_+$$
\end{description}
One then defines the {\em fibrations} to be the morphisms $p$ having the
right lifting property with respect to generating trivial cofibrations
(see e.g. \cite{Hovey}), the {\em (weak) equivalences} to be the local 
equivalences, the {\rm trivial fibrations} to be fibrations which are
also (weak) equivalences, and the {\em cofibrations} to be the morphisms
having the left lifting property with respect to trivial fibrations. 

Following \cite{KSB2} and using  \ref{12.6}, one proves that
the trivial fibrations can be equivalently described as morphisms
having the right lifting property with respect to the
following class of morphisms $(I)$:
\begin{description}
\item[$(I_a)$] $(\partial\Delta^n\times h_X)_+\subset (\Delta^n\times
h_X)_+$
\item[$(I_b)$] for $U\sr X$ open embedding,
$$(\Delta^n\times h_U\coprod_{\partial\Delta^n\times
h_U}\partial\Delta^n\times h_X)_+\sr (\Delta^{n}\times
h_X)_+$$
\end{description}
The maps of the form $(I)$ are called {\em generating cofibrations}. 

For $X$ and $Y$ pointed simplicial sheaves, one defines a pointed
simplicial set $S(X,Y)$ by 
$$S(X,Y)_n=Hom(X\wedge (\Delta^n)_+, Y)$$
Following \cite{KSB2}, one sees that the classes of cofibrations,
(weak) equivalences, fibrations, and $S$ are a simplicial closed model
structure in the sense of \cite{}. This has the following
consequences.
\begin{cor}
\llabel{cor12.4}
If $X$ is cofibrant and $Y$ fibrant, for any pointed simplicial set
$K$, one has in the relevant homotopy categories
$$Hom_{Ho}(X\wedge K, Y)=Hom_{Ho}(K, S(X,Y))$$
In particular, taking $k=(\Delta^0)_+$ one gets 
$$Hom_{Ho}(X,Y)=\pi_0 S(X,Y)$$
\end{cor}
\begin{cor}
\llabel{cor12.5}
If $X\sr Y$ is a cofibration and $Z$ a cofibrant object, then $X\wedge
Z\sr Y\wedge Z$ is a cofibration. 
\end{cor}
\begin{cor}
\llabel{cor12.6}
If $X$ is cofibrant and $Y$ is fibrant, then for any $Z$
\begin{equation}
\llabel{12.6.1}
Hom_{Ho}(Z,\uu{Hom}(X,Y))=Hom_{Ho}(Z\wedge X,Y)
\end{equation}
\end{cor}
In (\ref{12.6.1}), $\uu{Hom}(X,Y)$ is the pointed simplicial sheaf
with components the sheaves of homomorphisms from $X\wedge
(\Delta^n)_+$ to $Y$. 

We now apply this framework to prove the folloiwng criterion for
$\af$-locality.
\begin{proposition}
\llabel{12.7}
Let $F$ be a pointed simplicial sheaf on $(QP/G)$. If, as a
simplicial presheaf, $F$ is flasque, then $F$ is
$\af$-local if and only if, for any $U$ in $(QP/G)$,
$$F(U)\sr F(U\times\af)$$
is a weak equivalence.
\end{proposition} 
We recall that {\em $\af$-local} means that for any $Y$ one has the
following in the homotopy category
\begin{equation}
\llabel{12.7.1}
Hom_{Ho}(Y,F) = Hom_{Ho}(Y\wedge (h_{\af})_+, F)
\end{equation}
\begin{lemma}
\llabel{12.8}
A fibrant pointed simplicial sheaf is flasque.
\end{lemma}
\begin{proof}
The right lifting property of $F\sr *$ relative the morphisms $(J_b)$
means that for $U\subset X$ an open embedding, the morphism $F(X)\sr
F(U)$ is a Kan fibration. As $F$ is a sheaf, an uper distinguished
square 
$$
\begin{CD}
B @>>> Y\\
@VVV @VVV\\
A @>>> X
\end{CD}
$$
gives rise to a Cartesian square
$$
\begin{CD}
F(X) @>>> F(Y)\\
@VVV @VVV\\
F(A) @>>> F(B)
\end{CD}
$$
As $F(Y)\sr F(B)$ is a Kan fibration, this square is also homotopy
Cartesian.
\end{proof}
\begin{lemma}
\llabel{12.9}
 \ref{12.7} holds of the assumption ``$F$ is flasque'' is
replaced by the assumption ``$F$ is fibrant''.
\end{lemma}
\begin{proof}
``Only if'' $(I_a)$ for $n=0$ says that for any $U$, $(h_U)_+$ is
cofibrant. By \ref{cor12.4}, for any pointed simplicial set $K$, one
has
$$Hom_{Ho}((h_U)_+\wedge K,F)=Hom_{Ho}(K, S((h_U)_+, F))$$
and $S((h_U)_+,F)$ is just $F(U)$. If in (\ref{12.7.1}) we take
$Y=K\wedge (h_U)_+$, so that $Y\wedge (h_{\af})_+=K\wedge
(h_{U\times\af})_+$ we get
$$Hom_{Ho}(K, F(U\times \af))=Hom_{Ho}(K, F(U))$$
That this holds for any $K$ means that $F(U)\sr F(U\times\af)$ becomes
an isomorphism in the homomotopy category , hence is a weak
equivalence. 

\noindent
``If'' We apply \ref{cor12.6}. As $(h_{\af})_+$ is cofibrant and $F$
fibrant, 
$$Hom_{Ho}(Y\wedge (h_{\af})_+, F)=Hom_{Ho}(Y,
\uu{Hom}((h_{\af})_+,F))$$
and it suffice to show that 
$$F\sr \uu{Hom}((h_{\af})_+, F)$$
is a local  equivalence. This $\uu{Hom}$ is a simplicial sheaf
$U\mapsto F(U\times\af)$ and the claim follows.
\end{proof}
\begin{proof}
We can now finish the proof of  \ref{12.7}. Let $F\sr F'$
be a fibrant replacement of $F$. As $F$ and $F'$ are flasque, $F(U)\sr
F'(U)$ is a weak equivalence for any $U$. That all $F(U)\sr
F(U\times\af)$ be weak equivalences is hence equivalent to all
$F'(U)\sr F'(U\times\af)$ be weak equivalences, while $F$ is
$\af$-local if and only if $F'$ is.
\end{proof}

\subsection{$\bdl$-closed classes}\label{6}
The proof of the main theorem of this section will be postponed.
\begin{definition}
\llabel{3d1}
A class $S$ of morphisms of pointed simplicial sheaves is
$\Delta$-closed if 
\begin{enumerate}
\item (simplicial) homotopy equivalences are in $S$
\item if two of $f$, $g$ and $fg$ are in $S$ then so is the third
\item $S$ is stable by finite coproducts
\item if $F_{**}\sr G_{**}$ is a morphism of pointed bisimplicial
sheaves, and if all $F_{*p}\sr G_{*p}$ are in $S$, so is the diagonal
$\Delta(F)\sr \Delta(G)$.
\end{enumerate}
\end{definition}
\begin{definition}
\llabel{3d2}
The class $S$ is $\bdl$-closed if, in addition, it is stable by
arbitrary coproducts and colimits of sequences $(F_*\sr G_*)_n$ with
the property that, degree by degree, $(F_k)_n\sr (F_k)_{n+1}$
(resp. $(G_k)_n\sr (G_k)_{n+1}$) is isomorphic to an embedding
$A\subset A\coprod B$ of pointed sheaves.
\end{definition}
\begin{theorem}
\llabel{postponed}
The class of $\af$-equivalences is the $\bdl$-closure of the
union of the classes of
\begin{enumerate}
\item local  equivalences
\item morphisms $(U\times\af)_+\sr U_+$ for $U$ in $Sm/k$
\end{enumerate}
In particular, the class of $\af$-equivalences is $\bdl$-closed.
\end{theorem}

\subsection{The class of $\af$-equivalences is 
$\bdl$-closed}\llabel{13} 
The properties \ref{3d1}(1), \ref{3d1}(2), \ref{3d1}(3) are
clear. The last property is proved in  \ref{13.6}.
\begin{lemma}\llabel{13.1}
Let $A$ be a pointed simplicial set and $X$ a pointed simplicial
sheaf. If $X$ is fibrant and $\af$-local, then $X^A$ is $\af$-local. 
\end{lemma}
\begin{proof}
Because $X$ is fibrant, for any $Y$, one has in the homotopy category 
\begin{equation}
\llabel{13.1.1}
Hom_{Ho}(Y,X^A)=Hom_{Ho}(A\wedge Y,X)
\end{equation}
Applying this to $Y$ and $Y\wedge (\af_+)$ and using
$$(A\wedge Y)\wedge (\af_+)=A\wedge (Y\wedge (\af_+))$$
one deduces from the $\af$-locality of $X$ that of $X^A$.
\end{proof}
\begin{lemma}\llabel{13.2}
Let $f:K\sr L$ be a morphism of pointed simplicial sheaves and $A$ be
a pointed simplicial set. If $f$ is an $\af$-equivalence, then so
is $f\wedge A:X\wedge A\sr Y\wedge A$. 
\end{lemma}
\begin{proof}
One has to check that for any $\af$-local $X$ one has in the homotopy
category 
$$Hom_{Ho}(A\wedge L,X)=Hom_{Ho}(A\wedge K,X)$$
Replacing $X$ by a fibrant replacement, one may assume $X$
fibrant. Applying (\ref{13.1.1}) one is reduced to 
\ref{13.1}.
\end{proof}
\begin{lemma}
\llabel{13.3}
Let $f:K\sr L$ be a morphism of pointed simplicial sheaves. If $K$ and
$L$ are cofibrant, then $f$ is a $\af$-equivalence if and only if
for any fibrant $\af$-local $X$, the morphism of simplicial sets
$$S(L,X)\sr S(K,X)$$
is a weak equivalence.
\end{lemma}
\begin{proof}
``{\bf if}'' Taking $\pi_0$ one deduces from the assumptions that
$$Hom_{Ho}(L,X)\dsr Hom_{Ho}(K,X)$$
``{\bf Only if}'' The assumptions imply that $S(K,X)$ and $S(L,X)$ are
fibrant. For any pointed simplicial set $A$ one has 
$$Hom_{Ho}(A,S(K,X))=Hom_{Ho}(K\wedge A,X)$$
and similarly for $L$ and one applies  \ref{13.2}.
\end{proof}
\begin{proposition}
\llabel{13.4}
The coproduct of a family of $\af$-equivalences 
$$f_{\alpha}:X_{\alpha}\sr Y_{\alpha}$$
is an $\af$-equivalence.
\end{proposition}
\begin{proof}
There are commutative diagrams 
$$
\begin{CD} 
* @>>> X_{\alpha}' @>f_{\alpha}'>> Y_{\alpha}'\\
@VVV @VVV @VVV\\
* @>>> X_{\alpha}  @>f_{\alpha}>> Y_{\alpha}
\end{CD}
$$
where morphisms on the first line are cofibrations, and where the
vertical maps are local  equivalences, and similarly for $Y$. 
Replacing $X_{\alpha}$ (resp. $Y_{\alpha}$) by $X_{\alpha}'$
(resp. $Y_{\alpha}'$) we may and shall assume that the $X_{\alpha}$
and $Y_{\alpha}$ are cofibrant. The coproducts $\coprod X_{\alpha}$,
$\coprod Y_{\alpha}$ are then cofibrant too. One has 
$$S(\coprod X_{\alpha}, X)=\prod S(X_{\alpha}, X)$$
and similarly for the $Y_{\alpha}$, and one applies  \ref{13.3},
and the fact that a product of a family of weak equivalences of
fibrant pointed simplicial sets is a weak equivalence.
\end{proof}
\begin{proposition}
\llabel{13.5}
The colimit 
$$f:colim F_n\sr colim G_n$$
of an inductive sequence of $\af$-equivalences $f_n:F_n\sr
F_n$ is again an $\af$-equivalence.
\end{proposition}
\begin{proof}
One inductively constructs an inductive sequence of commutative squares
$$
\begin{CD}
F_n' @>f_n'>> G_n'\\
@VVV @VVV\\
F_n @>f_n>> G_n
\end{CD}
$$
in which the vertical maps are local  equivalences, the $F_n'$ and
$G_n'$ are cofibrant and the transition maps $F_n'\sr F_{n+1}'$,
$G_n'\sr G_{n+1}'$ are cofibrations. A colimit of local 
equivalences being a local  equivalence, it is sufficient to prove
the proposition for the sequence $(f_n')$. We hence may and shall
assume that $*\sr F_1\sr\dots\sr F_n\sr$ is a sequence of cofibrations
and similarly for the $*\sr G_1\sr\dots\sr G_n\sr$. The colimits $F$
and $G$ of those sequences are then cofibrant.

If $X$ is fibrant and $\af$-local, $S(G,X)\sr S(F,X)$ is the limit
of the sequence of weak equivalences
$$S(G_n,X)\sr S(F_n,X)$$
In the sequences $S(G_n,X)$ and $S(F_n,X)$ the transition maps are
fibrations of fibrant objects. It follows that the limit is again a
weak equivalence: the $\pi_i$ of the limit map onto the limit of
$\pi_i$, with fibers $(lim^1 \pi_{i+1})$-torsors. It remains to apply
 \ref{13.3}.
\end{proof}
\begin{proposition}
\llabel{13.6}
Let $F_{**}\sr G_{**}$ be a morphism of pointed bisimplicial
sheaves. If all $F_{p*}\sr G_{p*}$ are $\af$-equivalences, so is
$\Delta(F)\sr \Delta(G)$.
\end{proposition}
To prove  \ref{13.6} we will functorially attach to
$F_{**}$ an inductive sequence of pointed simplicial sheaves
$F^{(n)}$, whose colimit maps to $\Delta(F)$ by a local 
equivalence. We will then inductively prove that $F^{(n)}\sr G^{(n)}$
is an $\af$-equivalence, and apply  \ref{13.5}. We
begin with preliminaries to the construction of the $F^{(n)}$.
\begin{piece}\llabel{13.7}\rm
Let $\Delta_{inj}$ be the category of finite ordered sets
$\Delta^n=(0,\dots,n)$ and increasing  injective maps. For any category $C$
with finite coproducts, the forgetting functor
$$\omega:\Delta^{op}C\sr \Delta^{op}_{inj}C$$
has a left adjoint $\omega'$: ``formally adding degenerate
simplicies'': $(\omega'X)_n$ is the coproduct, over all $p$ and all
increasing surjective maps $s:\Delta^n\sr \Delta^p$, of copies of $X_p$ 
$$(\omega'X)_n=\coprod_{s} X_p$$
We define the {\em wrapping functor} $Wr:\Delta^{op}C\sr \Delta^{op}C$
as the composite $Wr:=\omega'\omega$. For $C$ the category of sets or
of pointed sets one has the folloiwng.
\end{piece}
\begin{lemma}
\llabel{13.8}
The adjunction map $a:Wr(X)\sr X$ is a weak equivalence.
\end{lemma}
\begin{proof}
We will prove it for $C$ the category of sets. The pointed case is
similar. The {\em fundamental groupoid} of $X$ is the category with
set of objects $X_0$, in which all maps are isomorphisms, and
universal for the property that 
\begin{enumerate}[(1)]
\item $\sigma\in X_1$ defines a morphism
$f(\sigma):\partial_1(\sigma)\sr \partial_0(\sigma)$
\item for $\tau\in X_2$,
$f(\partial_1\tau)=f(\partial_0\tau)f(\partial_2\tau)$. 
\end{enumerate}
One has $X_0=Wr(X)_0$. To handle $\pi_0$ and $\pi_1$ it suffice to
show that $a$ induces an isomorphism of fundamental groupoids. For any
$X$ and any $p\in X_0$, $f(s_0(p))$ is the identity of $p$. This
results from (2) applied to $s_0s_0(p)$ which gives
$$f(s_0(p))=f(s_0(p))f(s_0(p))$$
As generators of the fundamental groupoid, it hence suffices to take
non degenerate $\sigma\in X_1$. For relations, it then suffices to
take those coming from non degenerate $\tau\in X_2$: the degenerate
$\tau$ give nothing new.

If we apply this to $Wr(X)$, we find as set of generators $X_1$, and
relations indexed by $X_2$, the same relations as for $X$. 

The functor $Wr$ commutes with passage to connected components and to
passage to a covering. To handle higher $\pi_i$, this reduces us to
the case where $X$ (and hence $Wr(X)$) is connected and simply
connected. In this case it suffices to check that $a$ induces an
isomorphism in homology. It does because one has a commutative diagram   
$$
\begin{array}{ccc}
C_*(X)&\ii&C_*(Wr(X))/degeneracies\\
&\searrow&\dw\\
&&C_*(X)/degeneracies
\end{array}
$$
in which the first arrow is an isomorphism, the second the effect of
$a$ on homology, and the composite is a homotopy equivalence. 
\end{proof}
\begin{piece}\llabel{13.9}\rm
For $X$ a pointed simplicial sheaf, let $sk_n(X)$ be the
n-th skeleton of $X$ i.e. simplicial subsheaf of $X$ for which
$(sk_n(X))_p$ is the union of the images of the degeneracies $X_q\sr X_p$
for $q\le n$. One has push-out squares
\begin{equation}\llabel{13.9.1}
\begin{CD}
X_{n+1}\wedge (\partial\Delta^{n+1})_+ @>>> sk_n(Wr(X))\\
@VVV @VVV\\
X_{n+1}\wedge (\Delta^{n+1})_+ @>>> sk_{n+1}(Wr(X))
\end{CD}
\end{equation}
Let now $F$ be bisimplicial. Each $F_{n,\BB}$ is simplicial, and they
form a simplicial system of pointed simplicial sheaves. Let us apply
$Wr$ and $sk_n$ to the first variable i.e. to the simplicial sheaf
$F_{\BB,m}$ for each fixed $m$. We again have diagrams
(\ref{13.9.1}) and, taking the diagonal $\Delta$, one obtains
push-out squares:
\begin{equation}\llabel{13.9.2}
\begin{CD}
F_{n+1}\wedge (\partial\Delta^{n+1})_+ @>>> \Delta(sk_n(Wr(F)))\\
@VVV @VVV\\
F_{n+1}\wedge (\Delta^{n+1})_+ @>>> \Delta(sk_{n+1}(Wr(F)))
\end{CD}
\end{equation}
where $F_n$ now stands for the pointed simplicial sheaf
$F_{n,\BB}$. This way the simplicial sheaf $\Delta(Wr(F))$, which by
\ref{13.8} maps to $\Delta(F)$ by a local  equivalence, appears as
an inductive limit of (\ref{13.9.2}).
\end{piece}
\vskip 3mm

\noindent
{\bf Proof of  \ref{13.6}:} With the notations of
(\ref{13.9}) it suffice to show that the
$$\Delta\, sk_n\, Wr (F)\sr \Delta\, sk_n\, Wr(G)$$
(with $sk_n\, Wr$ applied in the first variable) are
$\af$-equivalences. We prove it by induction on $n$.

For $n=0$, $\Delta\, sk_0\, Wr (F)=F_{0}$, and $F_0\sr G_0$ is assumed to
be an $\af$-equivalence. From $n$ to $n+1$, we have a morphism of
push out squares
$$\mbox{\rm (\ref{13.9.2}) for $F$}\sr \mbox{\rm (\ref{13.9.2})
for $G$}$$
As $F_{n+1}\sr G_{n+1}$ is an $\af$-equivalence, by 
\ref{13.2}, so are its smash product with  $(\partial\Delta^{n+1})_+$
and $(\Delta^{n+1})_+$. It remain to apply the
\begin{lemma}
\llabel{13.11}
Suppose given a morphism of push out squares
$$
\begin{CD}
{\bf 1} @>>> {\bf 2}\\
@VVV @VVV\\
{\bf 3} @>>> {\bf 4}
\end{CD}
\,\,\,\,\,\,\,\longrightarrow\,\,\,\,\,\,\,
\begin{CD}
{\bf 1}' @>>> {\bf 2}'\\
@VVV @VVV\\
{\bf 3}' @>>> {\bf 4}'
\end{CD}
$$
which is an $\af$-equivalence in positions ${\bf 1}, {\bf 2}$ and ${\bf
3}$. If in each square the first vertical map is injective, then the
morphism of squares is an $\af$-equivalence in position ${\bf 4}$
as well.
\end{lemma}
\begin{proof}
Replacing the push out squares by push out squares of local 
equivalent objects, we may and shall assume that all objects
considered are cofibrant, and that the vertical maps are cofibrant.

If $X$ is fibrant applying $S(-,X)$ to each of the squares we get a
morphism of cartesian squares, of pointed simplicial sets in which the
vertical maps are cofibrations:
$$
\begin{CD}
S({\bf 4}, X) @>>> S({\bf 3}, X)\\
@VVV @VVV\\
S({\bf 2}, X) @>>> S({\bf 1}, X)
\end{CD}
\,\,\,\,\,\,\,\longleftarrow\,\,\,\,\,\,\,
\begin{CD}
S({\bf 4}', X) @>>> S({\bf 3}', X)\\
@VVV @VVV\\
S({\bf 2}', X) @>>> S({\bf 1}', X)
\end{CD}
$$
If $X$ is in addition $\af$-local, it is a weak equivalence in
positions ${\bf 1}, {\bf 2}$ and ${\bf 3}$, hence also in position
${\bf 4}$. By  \ref{13.3}, this proves  \ref{13.11},
finishing the proof of  \ref{13.6} as well as of the claim
that the class of $\af$-equivalences is $\bdl$-closed.
\end{proof}

\subsection{The class of $\af$-equivalences as a
$\bdl$-closure}\llabel{14} 
In this section we finish the proof of  \ref{postponed}. 
\begin{piece}\llabel{14.1}\rm
The homotopy push-out of a diagram
\begin{equation}\llabel{14.1.1}
Q\,\,\,\,\,\,\,:\,\,\,\,\,\,\,
\begin{CD}
K @>>> L\\
@VVV\\
M
\end{CD}
\end{equation}
is the push-out $K_Q$ of 
\begin{equation}\llabel{14.1.2}
Q\,\,\,\,\,\,\,:\,\,\,\,\,\,\,
\begin{CD}
K\vee K @>>> M\vee L\\
@VVV\\
K\wedge (\Delta^1)_+
\end{CD}
\end{equation}
where the vertical map $K\wedge (\Delta^0\coprod \Delta^0)_+\sr K\wedge
(\Delta^1)_+$ is induced by $\partial_1,\partial_0:\Delta^0\sr
\Delta^1$ mapping $\Delta^0$ to $0$ (resp. $1$) in $\Delta^1$. In the
case of simplicial sets, $|K_Q|$ maps to $|\Delta^1|=[0,1]$ with
fibers $|M|$ above $0$, $|L|$ above $1$, and $|K|$ above $(0,1)$. 

The homotopy push-out $K_Q$ is the diagonal of the bisimplicial object 
with columns $M\vee K^{\vee n}\vee L$ obtained by formally adding
degeneracies to 
$$K\dsr M\vee L$$
in $\Delta^{op}_{inj}\Delta^{op}(Sh_{\BB})$ (cf. \ref{13.7})
[$\partial_0$ maps $K$ to $L$, $\partial_1$ maps $K$ to $M$]. If
$f:Q\sr Q'$ is a morphism of diagrams (\ref{14.1.1}), the induced
morphism from $K_Q$ to $K_{Q'}$ is hence in the closure of the three
components of $f$ for the operation of finite coproduct and diagonal
(\ref{3d1}(3), (4)). 

A commutative square 
\begin{equation}\llabel{14.1.3}
\begin{CD}
K @>>> L\\
@VVV @VVV\\
M @>>> N
\end{CD}
\end{equation}
induces a morphism $K_Q\sr N$. 
\end{piece}
\begin{example}\llabel{ex14.2}\rm
Let $f:K\sr L$ be a morphism. The homotopy push-out of 
$$
\begin{CD}
K @>>> L\\
@VidVV\\
K
\end{CD}
$$
is the cylinder $cyl(f)$ of $f$. The morphisms 
$$L\sr cyl(f)\sr L$$
are homotopy equivalences. To check that the composite $cyl(f)\sr L\sr
cyl(f)$ is homotopic to the identity, one observes that $cyl(f)$ is
the push-out of 
$$
\begin{CD}
K @>>> L\\
@VVV\\
K\wedge (\Delta^1)_+
\end{CD}
$$
(the vertical map induced by $\partial_0:\Delta^0\sr \Delta^1$ mapping
$\Delta^0$ to $1$) and that the composite $cyl(f)\sr cyl(f)$ is
induced by $\Delta^1\sr \Delta^0\sr \Delta^1$, homotopic to the
identity by a homotopy fixing $1$. 

Similar arguments would show that the homotopy push out $cyl'(f)$ of 
$$
\begin{CD}
K @>id>> K\\
@VVV\\
M
\end{CD}
$$
is homotopic to $L$ by $L\sr cyl'(L)\sr L$. 
\end{example}
\begin{example}\llabel{ex14.3}\rm
In any category with finite coproducts, a {\em coprojection} is a map
isomorphic to the natural map $A\sr A\coprod B$ for some $A$ and
$B$. If in a push-out square of pointed simplicial sheaves
\begin{eq}\llabel{14.3.1}
\begin{CD}
K @>f>> L\\
@VVV @VVV\\
M @>>> N
\end{CD}
\end{eq}
the morphism $f$ is a coprojection: $L= K\vee A$, the square
(\ref{14.3.1}) is the coproduct of the squares
\begin{eq}\llabel{14.3.2}
\begin{CD}
K @>id>> K\\
@VVV @VVV\\
M @>id>> M
\end{CD}
\,\,\,\,\,\,\,
{\rm and}
\,\,\,\,\,\,\,
\begin{CD}
* @>>> A\\
@VVV @VVV\\
* @>>> A
\end{CD}
\end{eq}
and the resulting morphism $K_Q\sr N$ is a homotopy equivalence, being
the coproduct of the homotopy equivalences of Example \ref{ex14.2}
resulting from the two squares (\ref{14.3.2}). The same conclusion
applies if $K\sr M$ is a coprojection. 

A morphism of pointed simplicial sheaves $K\sr L$ is a termwise
coprojection if each $K_n\sr L_n$ is a coprojection of pointed
sheaves. Example: for any diagram (\ref{14.1.1}), the morphisms
$L,M\sr K_Q$ are termwise coprojections. For any morphism $f:K\sr L$,
this applies in particular to $K,L\sr cyl(f)$. 
\end{example}
\begin{proposition}\llabel{14.4}
If in a cocartesian square (\ref{14.1.3}) either $K\sr L$ or $K\sr
M$ is a termwise coprojection, then the resulting morphism from $K_Q$
to $N$ is in the $\Delta$-closure of the empty set of morphisms.
\end{proposition}
\begin{proof}
For each $n$, we have a cocartesian square of pointed sheaves 
$$
Q_n\,\,\,\,\,\,\, :\,\,\,\,\,\,\,
\begin{CD}
K_n @>>> L_n\\
@VVV @VVV\\
M_n @>>> N_n
\end{CD}
$$
Let us view it as a cocartesian square of pointed simplicial
sheaves. By  \ref{ex14.3}, it gives rise to a homotopy
equivalence $K_{Q_n}\sr N_n$. One concludes by observing that $K_Q\sr
N$ is the diagonal of this simplicial system of morphisms. 
\end{proof}
\begin{cor}\llabel{c14.5}
If in a cocartesian square (\ref{14.1.3}):
$$
\begin{CD}
K @>f>> L\\
@VgVV @VVg'V\\
M @>f'>> N
\end{CD}
$$
$f$ or $g$ is a termwise coprojection, then 
\begin{enumerate}
\item $f'$ is in the $\Delta$-closure of $\{f\}$
\item $g'$ is in the $\Delta$-closure of $\{g\}$
\end{enumerate}
\end{cor}
{\bf Proof of (1)}: The morphism of cocartesian squares
$$
Q'
\,\,\,\,\,\,\,:\,\,\,\,\,\,\,
\begin{CD}
K @>id>> K\\
@VVV @VVV\\
M @>id>> M
\end{CD}
\,\,\,\,\,\,\,
\longrightarrow
\,\,\,\,\,\,\,
Q \,\,\,\,\,\,\,:\,\,\,\,\,\,\,
\begin{CD}
K @>>> L\\
@VVV @VVV\\
M @>>> N
\end{CD}
$$
defines a commutative  square 
$$
\begin{CD}
K_{Q'} @>>> K_Q\\
@VVV @VVV\\
M @>>> N
\end{CD}
$$
in which the vertical maps are in the $\Delta$-closure of the empty
set by \ref{14.4}, while the first horizontal map is in the
$\Delta$-closure of $f$ by \ref{14.1}.
\vskip 3mm

\noindent
{\bf Proof of (2)}: One similarly uses
$$
Q'
\,\,\,\,\,\,\,:\,\,\,\,\,\,\,
\begin{CD}
K @>>> L\\
@VidVV @VVidV\\
K @>i>> L
\end{CD}
\,\,\,\,\,\,\,
\longrightarrow
\,\,\,\,\,\,\,
Q \,\,\,\,\,\,\,:\,\,\,\,\,\,\,
\begin{CD}
K @>>> L\\
@VVV @VVV\\
M @>>> N.
\end{CD}
$$
\begin{piece}\llabel{14.5}\rm
A pointed simplicial sheaf $F_{\BB}$ is {\em reliably compact} if it
coincides with its $n$-skeleton for some $n$ and each $F_i$ is compact
in the sense that the functor $Hom(F_i,-)$ commutes with filtering
colimits. This property is stable by $F_{\BB}\sr F_{\BB}\wedge K$ for
$K$ a finite pointed simplicial set (finite number of non degenerate
simplices) and implies that $F_{\BB}$ is compact.
\end{piece}
\begin{cons}
\llabel{14.6}\rm
Let $E$ and $N$ be classes of morphisms such that 
\begin{enumerate}[(a)]
\item sources and targets are reliably compact
\item each $f$ in $N$ is a termwise coprojection
\end{enumerate}
We will construct a functor $Ex$ from pointed simplicial sheaves
to pointed simplicial sheaves and a morphism $Id\sr Ex$ such that:
\begin{enumerate}[(i)]
\item For any $F$, $F\sr Ex(F)$ is in the $\bdl$-closure of $E$
\item If $f:K\sr L$ is in $E$, the morphism 
\begin{eq}\llabel{14.6.1}
S(L,Ex(X))\sr S(K,Ex(X)) 
\end{eq}
is a weak equivalence.
\item If $f:K\sr L$ is in $N$, the morphism (\ref{14.6.1}) is a Kan
fibration. 
\end{enumerate}
\comment{For $f:K\sr L$ in $E$, the factorization of $f$ as $K\sr cyl(f)\sr L$
induces a factorization of (\ref{14.6.1})
$$S(L,Ex(F))\sr S(cyl(f), Ex(F))\sr S(K, Ex(F))$$}
Let us factorize $f:K\sr L$ as $K\sr cyl(f)\sr L$. As the second map
is a homotopy equivalence, the first is in the $\Delta$-closure of
$E$. In the corresponding factorization of (\ref{14.6.1}):
$$S(L,Ex(F))\sr S(cyl(f), Ex(F))\sr S(K, Ex(F))$$
the first map is a homotopy equivalence.  To obtain (ii), it hence
suffices that $S(cyl(f),Ex(F))\sr S(K,Ex(F))$ be a weak equivalence. 

Replacing each $f:K\sr L$ in $E$ by the corresponding $K\sr cyl(f)$,
this reduces us to the case where 
\vskip 2mm

\noindent
(c) each $f$ in $E$ is a termwise coprojection,
\vskip 2mm

\noindent
and we will construct in this case a functor $Ex$ such that  
\vskip 2mm

\noindent
(iii)$^*$ for $f$ in $E$, (\ref{14.6.1}) is a trivial fibration.
\vskip 2mm

The conditions (ii), (iii)$^*$ are lifting properties:
\vskip 2mm

\noindent
for $f$ in $E$, in squares:
$$
\begin{CD}
\partial\Delta^n_+ @>>> S(L,Ex(F))\\
@VVV @VVV\\
\Delta^n_+ @>>> S(K,Ex(F))
\end{CD}
$$
\vskip 2mm

\noindent
for $f$ in $N$, in squares:
$$
\begin{CD}
(\Lambda^n_k)_+ @>>> S(L,Ex(F))\\
@VVV @VVV\\
\Delta^n_+ @>>> S(K,Ex(F))
\end{CD}
$$
In the first case, the data are morphisms $\Delta^n_+\wedge K\sr Ex(F)$
and $\partial\Delta^n_+\wedge L\sr Ex(F)$ agreeing on
$\partial\Delta^n_+\wedge K$ i.e. a morphism 
$$(\Delta^n_+\wedge K)\coprod_{\partial\Delta^n_+\wedge
K}(\partial\Delta^n_+\wedge L)\sr Ex(F)$$
and we want it to extend to $\Delta^n_+\wedge L$. Similarly in the
second case, with $\partial\Delta^n$ replaced by $\Lambda^n_k$:
\vskip 2mm

\noindent
for $f$ in $E$:
\begin{eq}\llabel{14.6.2}
\begin{CD}
(\Delta^n_+\wedge K)\coprod_{\partial\Delta^n_+\wedge
K}(\partial\Delta^n_+\wedge L) @>>> Ex(F)\\
@VVV @VVV\\
\Delta^n_+\wedge L @>>> *
\end{CD}
\end{eq}
\vskip 2mm

\noindent
for $f$ in $N$:
\begin{eq}\llabel{14.6.3}
\begin{CD}
(\Delta^n_+\wedge K)\coprod_{(\Lambda^n_k)_+\wedge
K}((\Lambda^n_k)_+\wedge L) @>>> Ex(F)\\
@VVV @VVV\\
\Delta^n_+\wedge L @>>> *
\end{CD}
\end{eq}
The left vertical maps are termwise coprojections, and their sources
are compact. One now uses the standard trick of defining $Ex(F)$ as
the inductive limit of the iterates of functors $F\sr T(F)$, where
$T(F)$ is deduced from $F$ by push out, simultaneously for all 
$$(\Delta^n_+\wedge K)\coprod_{\partial\Delta^n_+\wedge
K}(\partial\Delta^n_+\wedge L)\sr Ex(F)\,\,\,\,\,\,\,(f:K\sr L \,\,\,{\rm
in} \,\,\,E)$$
and
$$(\Delta^n_+\wedge K)\coprod_{(\Lambda^n_k)_+\wedge
K}((\Lambda^n_k)_+\wedge L)\sr Ex(F)\,\,\,\,\,\,\,(f:K\sr L \,\,\,{\rm
in} \,\,\,N)$$
The push out is by
$$\bigvee ({\rm sources}) \sr \bigvee (\Delta^n_+\wedge L)$$
a morphism which is a termwise coprojection. By \ref{14.5}, to check
that the resulting $F\sr Ex(F)$ is in the $\bdl$-closure of $E$, it
suffices to check that the left vertical morphism in (\ref{14.6.2})
(resp. (\ref{14.6.3})) is in the $\Delta$-closure of $E$ (resp. of the
empty set). 

For (\ref{14.6.2}), this is the map marked $\bf 3$ in
$$
\begin{CD}
\partial\Delta^n_+\wedge K @>{\bf 1}>> \partial\Delta^n_+\wedge L\\
@VVV @VVV\\
\Delta^n_+\wedge K @>{\bf 2}>> \dots @>{\bf 3}>> \Delta^n_+\wedge L
\end{CD}
$$
The morphisms $\bf 1$ and ${\bf 3}\circ {\bf 2}$ are in the
$\Delta$-closure of $E$. So is $\bf 2$ by \ref{14.5} and one applies
the 2 out of 3 property .

For (\ref{14.6.3}), the diagram is
$$
\begin{CD}
(\Lambda^n_k)_+\wedge K @>>> (\Lambda^n_k)_+\wedge L\\
@VV{\bf 1}V @VV{\bf 2}V\\
\Delta^n_+\wedge K @>>> \dots @>{\bf 3}>> \Delta^n_+\wedge L
\end{CD}
$$
with $\bf 1$ and ${\bf 3}\circ {\bf 2}$ in the $\Delta$-closure of the
empty set. Indeed, $\Lambda^n_k$ and $\Delta^n$ are both
contractible. 
\end{cons}
\begin{remark}\llabel{14.7}\rm
Let $P$ be a property of pointed simplicial schemes stable by
coproduct, and suppose that
\begin{enumerate}[(a)]
\item for $f:K\sr L$ in $E$, the $K_n$ and $L_n$ have property $P$
\item for $f:K\sr L$ in $N$, $f$ is in degree $n$ isomorphic to the
natural map 
$K_n\sr K_n\vee A$ for some $A$ having property $P$.
\end{enumerate}
The functor $Ex$ constructed in \ref{14.6} is then such that for any
$K$, each morphism $K_n\sr Ex(K)_n$ is isomorphic to some $K_n\sr
K_n\vee A$ where $A$ has property $P$. In particular, if the $K_n$
have property $P$, so have the $Ex(K)_n$.
\end{remark}
\begin{piece}[Proof of \ref{postponed}]\label{14.8}\rm
We apply construction \ref{14.6}, on the site $QP/G$, taking for
$E$ and $N$ the following classes.
\vskip 2mm
$E$: For any $X$ in the site, the morphism 
\begin{eq}\llabel{14.7.1}
(X\times\af)_+\sr X_+
\end{eq}
and for any upper distinguished square 
\begin{eq}\llabel{14.7.2}
\begin{CD}
B @>>> Y\\
@VVV @VVV\\
A @>>> X,
\end{CD}
\end{eq}
the morphism
\begin{eq}\llabel{14.7.3}
(K_Q)_+\sr X_+
\end{eq}
\vskip 2mm
$N$: For any $X$ in the site,
\begin{eq}\llabel{14.7.4}
(\emptyset)_+\sr X_+
\end{eq}
\vskip 2mm

If a pointed simplicial sheaf $G$ is of the form $Ex(F)$, that
(\ref{14.7.4}) is in $N$ ensures that each $G(X)$ is Kan. That
(\ref{14.7.3}) is in $E$ ensures that for each upper distinguished
square (\ref{14.7.2}), the morphism 
$$G(X)\sr S((K_{Q})_+, G)$$
is a weak equivalence. As each $G(Y)$ is Kan, $S((K_Q)_+,G)$ is the
homotopy fiber product of $G(A)$ over $G(B)$, and 
$$
\begin{CD}
G(X) @>>> G(A)\\
@VVV @VVV\\
G(Y) @>>> G(B)
\end{CD}
$$
is homotopy cartesian: $G$ is flasque.

Further, as (\ref{14.7.1}) is in $E$, for each $X$,
$$G(X)\sr G(X\times\af)$$
is a weak equivalence: by \ref{12.6}, $G$ is $\af$-local.

Suppose now that $f:F\sr G$ is a $\af$-equivalence. In the
commutative diagram 
$$
\begin{CD}
F @>>> G\\
@VVV @VVV\\
Ex(F) @>>> Ex(G)
\end{CD}
$$
the vertical maps are in the $\bdl$-closure of the morphisms
(\ref{14.7.1}) and (\ref{14.7.2}), the later being local 
equivalences. In particular, they are $\af$-equivalences and
$Ex(F)\sr Ex(G)$ is an $\af$-equivalence between $\af$-local
objects , hence is a local  equivalence. It follows that $f$ is in
the required $\bdl$-closure, proving \ref{postponed}.
\end{piece}
The functor $Ex$ used introduced in \ref{14.8} can also be used to
prove the following lemma.
\begin{lemma}
\llabel{new6} If $F^{(i)}\sr F^{(j)}$ is a filtering system of
$\af$-equivalences, then $F^{(n)}\sr colim_{i}F^{(i)}$ is again an
$\af$-equivalence. 
\end{lemma}
\begin{proof}
Consider the square:
$$
\begin{CD}
F^{(n)} @>>> colim_{i}F^{(i)}\\
@VVV @VVV\\
Ex(F^{(n)}) @>>> Ex(colim_{i}F^{(i)}).
\end{CD}
$$
Since the functor $Ex$ commutes with filtering colimits, the bottom
arrow is a filtering colimit of local equivalences, hence a local
equivalence. The vertical maps are $\af$-equivalences, hence the top
map is an $\af$-equivalence.
\end{proof}

\subsection{One more characterization of equivalences}
Denote by $[QP/G]_+$ the full subcategory in the category of pointed
sheaves on $QP/G$ generated by all coproducts of sheaves of the form
$(h_X)_+$. 
\begin{theorem}\llabel{conv}
The class of local equivalences (resp. $\af$-equivalences) in
$\Delta^{op}[QP/G]_+$ is the smallest class $W$ which contains
morphisms $(K_Q)_+\sr X_+$ for $Q$ upper distinguished and has the
following properties:
\begin{enumerate}
\item simplicial homotopy equivalences (resp. and $\af$-homotopy
equivalences) are in $W$
\item if two of $f$, $g$ and $fg$ are in $W$ then so is the third
\item if $F^{(i)}\sr F^{(j)}$ is a filtering system of termwise
coprojections in $W$, then $F^{(n)}\sr colim_{i}F^{(i)}$ is again in
$W$
\item if $F_{**}\sr F'_{**}$ is a morphism of bisimplicial objects,
and if all $F_{*p}\sr F'_{*p}$ are in $W$, so is the diagonal
$\Delta(F)\sr \Delta(F')$.
\end{enumerate}
\end{theorem}
The proof is given in \ref{convp}
\begin{lemma}\llabel{15.3}
If the morphism $f:F\sr G$ is such that, for each $U$, $F(U)\sr G(U)$
is a weak equivalence, and if the $F_n$ and $G_n$ are all of the form
$(\coprod h_{U_i})_+$, then $f$ is in the $\bdl$-closure of the empty
set.
\end{lemma}
The proof will use the following construction.
\begin{cons}\llabel{15.4}\rm
Let $C$ be a category, and let $C_0$ be a set of objects of $C$, such
that any isomorphism class has a representative in $C_0$. Let $i_*$ be
the functor which to a presheaf of pointed sets on $C$ attaches the
family of pointed sets $(F(U))_{U\in C_0}$. It has a left adjoint $i^*$:
$${\rm family}\,\,(A_U)_{U\in C_0}\mapsto \bigvee_{U} ((h_U)_+\wedge
A_U)=$$
$$=(\mbox{\rm disjoint sum over the $U\in C_0$ and $(A_U-*)$ of
$h_U$})_+$$
If $C_0$ is viewed as a category whose only morphisms are identities,
the natural functor
$$i:C_0\sr C$$
defines a morphism of sites $C\sr C_0$, both endowed with the trivial
topology (any presheaf a sheaf), and $i_*$, $i^*$ are the
corresponding direct and inverse image of pointed sheaves. 

By a general story valid for any pair of adjoint functors, for any
pointed presheaf $F$ on $C$, the $(i^*i_*)^{n+1}(F)$ form a pointed
simplicial presheaf $R(F)$ augmented to $F$:
$$a:R(F)\sr F$$
Further
\begin{enumerate}[(a)]
\item If $F$ is of the form $i^*A$, i.e. of the form $(\coprod
h_{U_i})_+$, $a$ is a homotopy equivalence
\item For any $F$, $i^*(a)$ is a homotopy equivalence: for each $U$ in
$C$, $R(F)(U)\sr F(U)$ is a homotopy equivalence
\end{enumerate}
For a simplicial presheaf $F$ we define 
$$R(F)=\Delta(\mbox{\rm simplicial system of $R(F_p)$})$$
\end{cons}
\begin{piece}\llabel{15.5}[Proof of \ref{15.3}]:\rm Let us say that
$X$ in $QP/G$ is connected if it is not empty and is not a disjoint
union: $G$ should act transitively on the set of connected components
of $X$. Let $C\subset QP/G$ be the full subcategory of connected
objects. A sheaf $F$ on $QP/G$ is determined by its restriction to
$C$. Indeed, $F(\coprod X_i)=\prod F(X_i)$. To apply \ref{15.4}, we
will use this remark to identify the category of sheaves on $QP/G$ to
a full subcategory of the category of presheaves on $C$. For any $C_0$
as in \ref{15.4}, the functor $i^*$ takes values in sheaves, that is
in the restriction of sheaves to $C$. Indeed, for $X$ connected,
$(\coprod h_{U_i})_+(X)$ is the same, whether $\coprod$ and $_+$ are
taken in the sheaf or in the presheaf sense. 

Fix $f:F\sr G$ as in \ref{15.3}. For each $n$, the assumption on $F_n$
ensures that $R(F_n)\sr F_n$ is in the $\Delta$-closure of the empty
set, and similarly for $G$. 

For each (connected) $U$, the morphism of pointed simplicial sets
$F(U)\sr G(U)$ is a weak equivalence, hence in the $\bdl$-closure of
the empty set. It follows that $i^*i_*(F)\sr i^*i_*(G)$: the $\vee$
over $C_0$ of the
$$(h_U)_+\wedge F(U)\sr (h_U)_+\wedge G(U)$$
is in the $\bdl$-closure of the empty set. Iterating one finds the
same for $(i^*i_*)^n(F)\sr (i^*i_*)^n(G)$, and $R(F)\sr R(G)$ is in
this $\bdl$-closure too. It remains to apply the two out of three
property to
$$
\begin{CD}
R(F) @>>> R(G)\\
@VVV @VVV\\
F @>>> G
\end{CD}
$$
\end{piece}
\begin{lemma}
\llabel{15.6}
If $f:F\sr G$ is a local  equivalence and if the $F_n$ and $G_n$
are all of the form $(\coprod h_{U_i})_+$, then $f$ is in the
$\bdl$-closure of the $(K_{Q})_+\sr X_+$ for $Q$ upper distinguished. 
\end{lemma}
\begin{proof}
We will use the construction \ref{14.6} for $E$ the class of morphisms 
$(K_{Q})_+\sr X_+$ for $Q$ upper distinguished, and for $N$ the class
of morphisms $*\sr X_+$. By \ref{14.7}, if the $F_n$ are of the form
$(\coprod h_{U_i})_+$, so are the $Ex(F)_n$. In the commutative
diagram 
$$
\begin{CD}
F @>>> G\\
@VVV @VVV\\
Ex(F) @>>> Ex(G)
\end{CD}
$$
the vertical maps are in the required $\Delta$-closure. They are in
particular local equivalences and so is $Ex(f)$. One verifies as in
\ref{14.8} that $Ex(F)$ and $Ex(G)$ are flasque. By \ref{14.6}(ii),
for each $U$, $Ex(f)(U)$ is a weak equivalence, and it remains to
apply \ref{15.3} to Ex(f).
\end{proof}
\begin{lemma}
\llabel{new1}
If $f:F\sr G$ is an $\af$-equivalence and if the $F_n$ and $G_n$
are all of the form $(\coprod h_{U_i})_+$, then $f$ is in the
$\bdl$-closure of the $(K_{Q})_+\sr {X}_+$ for $Q$ upper distinguished
and $(X\times\af\sr X)_+$ for $X\in QP$.  
\end{lemma}
\begin{proof}
Similar to the proof of \ref{15.6}.
\end{proof}
\begin{piece}\llabel{new2}\rm
Since for any simplicial sheaf $F$ the map $R(F)\sr F$ is a local
equivalence Lemmas \ref{15.6} and \ref{new1} imply that for any local
(resp. $\af$-) equivalence $f:F\sr G$, the morphism $R(f)$ belongs to
the $\bdl$-closure of the $(K_{Q})_+\sr X_+$ for $Q$ upper
distinguished (resp. the $(K_{Q})_+\sr X_+$ for $Q$ upper
distinguished and $(X\times\af\sr X)_+$ for $X\in QP$).
\end{piece}
\begin{piece}{Proof of \ref{conv}}\llabel{convp}\rm: We consider only
the case of $\af$-equivalences. Proposition \ref{13.6} and Lemma
\ref{new6} imply that $\af$-equivalences contain the class $W$. In
view of Lemma \ref{new1} it remains to see that $W$ is
$\bdl$-closed. The only condition to check is that it is closed under
coproducts. Let
$$f_{\alpha}:F^{(\alpha)}\sr H^{(\alpha)},\,\,\,\,\,\, \alpha\in
A$$
be a family of morphisms in $W$. For a finite subset $I$ in $A$ set
$$\Phi_I=(\coprod_{\alpha\in I}
H^{(\alpha)})\coprod(\coprod_{\alpha\in A-I} F^{(\alpha)})$$
For $I\in J$ we have a morphism $\Phi_I\sr \Phi_J$ and the map
$\coprod_{f_{\alpha}}$ is isomorphic to the map 
$$\Phi_{\emptyset}\sr colim_{I\subset A} \Phi_I$$
It remains to show that $\Phi_I\sr \Phi_{I\cup\{\alpha}$ is in
$W$. This morphism is of the form $Id_H\coprod (f:F\sr F')$ where $f$
is in $W$. Using the fact that $W$ is closed for diagonals we reduce
to the case $H=\coprod (h_U)_+$. Using the same reasoning as above we
further reduce to the case $H=(h_U)_+$. 

Consider the class of $f$ such that $Id_{(h_U)_+}\coprod f$ is in
$W$. This class clearly contains morphisms $(K_Q)_+\sr X_+$, has the
two out of three property and is closed under filtering colimits. It
also contains simplicial homotopy equivalences. It contains morphisms
of the form $p_+:(X\times\af)_+\sr X_+$ because such morphisms are
$\af$-homotopy equivalences.
\end{piece}

\section{Solid sheaves}

\subsection{Open morphisms and solid morphisms of sheaves\label{8}\label{9}}
We fix $S$ and $G$ as in Section \ref{10}, and will work in
$(QP/G)_{Nis}$. The story could be repeated in $(Sm/S)_{Nis}$. 
\begin{definition}
\llabel{d8.1} A morphism of sheaves $f:F\sr G$ is open if it is
relatively representable by open embeddings, i.e. if for any morphism
$u:h_X\sr G$ (that is, $u\in G(X)$, $X$ in $QP/G$), the fiber product
$F\times_G h_X$ mapping to $h_X$ is isomorphic to $h_U\sr h_X$ for $U$
a $G$-stable open subset of $X$.
\end{definition} 
In other words: $f$ should identify $F$ with a subsheaf of $G$ and,
for any $s\in G(X)$, there is $U$ open in $X$ and $G$-stable such that
the pull-back of $s$ with respect to $Y\sr X$ is in $F(Y)$ if and only
if $Y$ maps to $U$.

The property ``open'' is stable under composition. It is also stable
by pull-back: if in a cartesian square
\begin{equation}
\llabel{8.1.1}
\begin{CD}
F' @>f'>> G'\\
@VVV @VVuV\\
F @>f>> G
\end{CD}
\end{equation}
$f$ is open, then $f'$ is open. This follows from transitivity of
pull-backs. Conversely, if $f'$ is open and $u$ is an epimorphism,
then $f$ is open.  Indeed,
\begin{lemma}
\llabel{8.1.2}
For $F\sr h_X$ a morphism, the property that $F$ is represented by $U$
open in $X$ is local on $X$ (for the Nisnevich topology).
\end{lemma}
\begin{proof}
Suppose that the $X_{\alpha}$ cover $X$, and that each $F_{\alpha}=
F\times_{h_X} h_{X_{\alpha}}$ is represented by $U_{\alpha}\subset
X_{\alpha}$. For $Y\sr X_{\alpha}$, $F_Y:=F\times_{h_X} h_Y$ is then
represented by $U_Y\subset Y$ with $U_Y$ the inverse image of
$U_{\alpha}$. By descent for open embedding, the $U_{\alpha}$ come
from some $U\subset X$, we have locally on $X$ an isomorphism $F\ii
h_U$ and by descent for isomorphisms of sheaves one has $F\ii h_U$. 
\end{proof}
Given a square of the form (\ref{8.1.1}) with $f'$ open and $u$ an
epimorphism, if $s$ is in $G(X)$, $s$ can locally be lifted to a
section of $G'$. As $f'$ is open, it follows that locally on $X$,
$F\times_G h_X$ is represented by an open subset. Applying 
\ref{8.1.2} one concludes that $f$ is open. The same argument shows
that if we have cartesian diagrams 
$$
\begin{CD}
F_{\alpha}' @>f_{\alpha}'>>G_{\alpha}'\\
@VVV @VVu_{\alpha}V\\
F @>f>> G
\end{CD}
$$
with each $f'_{\alpha}$ open and $\coprod u_{\alpha}:\coprod
G_{\alpha}'\sr G$ onto, then $f$ is open.
\begin{proposition}
\llabel{8.1.3}
The property ``open'' is stable by push-outs.
\end{proposition}
\begin{proof}
Suppose
$$
\begin{CD}
F @>f>> G\\
@VVV @VVV\\
F' @>f'>> G'
\end{CD}
$$
is a cocartesian diagram, with $f$ open. In particular $f$ is a
monomorphism, and it follows that $f'$ is a monomorphism and that the
square is cartesian as well. The morphism $F'\coprod G\sr G'$ is
onto. The pull-back of $f'$ by $f':F'\sr G'$ is an isomorphism ($f'$
being a monomorphism) hence open. The pull-back of $f'$ by $G\sr G'$
is just $f$, open by assumption. It follows that $f'$ is open.
\end{proof}
We now fix a class $C$ of open embeddings $U\sr V$ in $(QP/G)$. We
require the following stabilities
\begin{cond}\rm\llabel{cond1}
\begin{enumerate}
\item If $U\sr U'\sr V$ are open embeddings and if $U\sr V$ is in $C$,
so is $U'\sr V$.
\item If $U\sr V$ is an open embedding in $C$, and if $f:V'\sr V$ is
etale, with $f^{-1}(Y)\subset Y$ for $Y$ the complement of $U$ in $V$,
then $f^{-1}(U)\sr V'$ is in $C$.
\end{enumerate}
\end{cond}
The classes $C$ we will have to consider are the following
\begin{enumerate}
\item The open embeddings $U\sr V$ with $V$ smooth.
\item The open embeddings $U\sr V$ with $V$ smooth such that the
action of $G$ is free on $V-U$. Equivalently: $V$ is the union of $U$
and the open subset on which the action of $G$ is free.
\item When working in $(Sm/S)$: all open embeddings
\end{enumerate}
\begin{definition}
\llabel{d8.2} A morphism $f:F\sr G$ is $C$-solid if it is a composite
$F=F_0\sr F_1\sr\dots\sr F_n=G$ where each $F_i\sr F_{i+1}$ is deduced
by push-out from some $h_U\sr h_X$, $U\subset X$ in $C$.

A sheaf $F$ is solid if $\emptyset\sr F$ is $C$-solid. 

In the pointed context, a pointed sheaf is (pointed) $C$-solid if the
morphism $pt\sr F$ is $C$-solid.
\end{definition}
\begin{example}
\rm For $U$ open in $X$, let $h_{X/U}$ be the sheaf $h_X$, with the
subsheaf $h_U$ contracted to a point $p$. If $U\sr X$ is in $C$, then
$p:pt\sr h_{X/U}$ is solid: it is the push-out of $h_U\sr h_X$ by
$h_U\sr pt$. Thom spaces are of this form: starting from a vector
bundle $V$ on $Y$, one contracts, in the total space of this vector
bundle, the complement of the zero section to a point.
\end{example} 
The class of solid morphisms is the smallest class closed by
compositions and push-outs which contains all $h_U\sr h_X$ for
$U\subset X$ in $C$. By  \ref{8.1.3} solid morphisms are
open.

For $F\sr G$ a monomorphism of sheaves, define $G/F$ to be the pointed
sheaf obtained by contracting $F$ to a point: one has a cocartesian
square 
$$
\begin{CD}
F @>>> G\\
@VVV @VVV\\
pt @>>> F/G
\end{CD}
$$
By transitivity of push-out, any cocartesian diagram  
$$
\begin{CD}
F @>>> G\\
@VVV @VVV\\
F' @>>> G'
\end{CD}
$$
induces an isomorphism $G/F\sr G'/F'$. 
\begin{proposition}
\llabel{b1}
A morphism of sheaves $f:F\sr G$ is $C$-solid if and only if it is a
composite $F=F_0\sr F_1\sr\dots\sr F_n=G$ of monomorphisms where each
$F_i/F_{i+1}$ is isomorphic to some $h_{V/U}=h_V/h_U$ for $U\subset V$
in $C$. 
\end{proposition}
\begin{proof}
If a morphism $F\sr G$ is deduced by push-out from $U\sr V$, $G/F$ is
isomorphic to $h_{V/U}$. From this, ``only if'' results. Conversly, if
we have
\begin{equation}
\llabel{bstar}
\begin{CD}
F @>>> G\\
@VVV @VVV\\
* @>>> h_{V/U}
\end{CD}
\,\,\,\,\,\,,\,\,\,\,\,\,
\begin{CD}
h_{U} @>>> h_V\\
@VVV @VVV\\
* @>>> h_{V/U}
\end{CD}
\end{equation}
cocartesian, and if $h_V\sr h_{V/U}$ lifts to $G$, then $F\sr G$ is
deduced by push-out from $h_U\sr h_V$. Indeed, the diagrams
(\ref{bstar}) being cartesian as well as cocartesian, we have a
cartesian 
$$
\begin{CD}
h_U @>>> h_V\\
@VVV @VVV\\
F @>>> G
\end{CD}
$$
If $G_1$ is deduced from $h_U\sr h_V$ by push-out:
$$
\begin{CD}
h_U @>>> h_V @>=>> h_V\\
@VVV @VVV @VVV\\
F @>>> G_1 @>>> G
\end{CD}
$$
then $G_1/F\cong G/F$ and it follows that $G_1=G$. 

Let us suppose only that we have $v:\tilde{V}\sr V$ etale, inducing an
isomorphism from $\tilde{V}-v^{-1}(U)$ to $V-U$ and a lifting of
$h_{\tilde{V}}\sr h_{V/U}$ to $G$. If $\tilde{U}:=v^{-1}(U)$,
$h_{\tilde{V}/\tilde{U}}\sr h_{V/U}$ is an isomorphism. This is most
easily checked by applying the fiber functors defined by a $G$-local
henselian $Y$:  a morphism $Y\sr V$, if it does not map to $U$, lifts
uniquely to a morphism to $\tilde{V}$. The assumptions made hence imply
that $F\sr G$ is a push-out of $h_{\tilde{U}}\sr h_{\tilde{V}}$. Note
that by the second stability property of $C$, $\tilde{U}\sr \tilde{V}$
is in $C$. 

We will reduce the proof of ``if'' to that case. We have to show that
if a monomorphism $f:F\sr G$ is such that $G/F\cong h_{V/U}$ with
$U\sr V$ in $C$, then $f$ is $C$ solid. The cocartesian square  
\begin{equation}\llabel{secstar}
\begin{CD}
F @>>> G\\
@VVV @VVV\\
pt @>>> F/G
\end{CD}
\end{equation}
induces an epimorphism $pt\coprod G\sr h_{V/U}$. The natural section
of $h_{V/U}$ on $V$ can hence locally be lifted to $pt$ or to $G$: for
some filtration $\emptyset=Z_n\subset \dots \subset Z_1\subset Z_0=V$
of $V$ by closed equivariant subschemes, we have etale maps
$\phi_i:Y_i\sr V$ with a (equivariant) section over $Z_i-Z_{i+1}$, and
a lifting of $h_{Y_i}\sr h_{V/U}$ to $pt$ or to $G$. Note $V_i:=
V-Z_{i+1}$. We may: 
\begin{enumerate}
\item start with $V_0=U$, taking $Y_0=U$: here the lifting is to $pt$
\item assume $V_i\ne V_{i+1}$; the succeeding liftings then cannot be
to $pt$: they must be to $G$ 
\item shrink $Y_i$, first so that it maps to $V_i$, next so that it
induces an isomorphism from $Y_i-\phi_i(V_{i-1})$ to $Z_i-Z_{i+1}$
\end{enumerate}
As $F\sr G$ is a monomorphism the cocartesian (\ref{secstar}) is
cartesian as well. The composition
$$pt = h_{V_0/U}\sr h_{V_1/U}\sr \dots \sr h_{V/U}$$
gives by pull-back a factorization of $F\sr G$ as 
$$F\sr F_1\sr\dots\sr G$$
with each 
$$
\begin{CD}
F_i @>>> F_{i+1}\\
@VVV @VVV\\
h_{V_i/U} @>>> h_{V_{i+1}}/U
\end{CD}
$$
cartesian and cocartesian, hence $F_{i+1}/F_i\cong
h_{V_{i+1}}/h_{V_i}$. Further, the morphism $\phi_{i+1}:h_{Y_{i+1}}\sr
h_{V_{i+1}}\sr h_{V/U}$ lifts to $G$, hence $h_{Y_{i+1}}\sr
h_{V_{i+1}/U}$ lifts to $F_{i+1}$. It follows that $F_i\sr F_{i+1}$ is
a push-out of $\phi_{i+1}^{-1}(V_i)\sr Y_{i+1}$, which is in $C$, and
solidity follows. 
\end{proof}
\begin{remark}
\rm\llabel{nrm}
Another formulation of  \ref{b1} is: a morphism $F\sr G$ is
$C$-solid if and only if the pointed sheaf $G/F$ is an iterated
extension of $h_{V/U}$'s with $U\sr V$ in $C$, in the sense that there
are morphisms 
$$pt= H_0\sr \dots\sr H_n=G/F$$
with each $H_{i+1}/H_i$ of the form $h_{V/U}$.
\end{remark}
\begin{proposition}
\llabel{8.3}
If $f:F\sr G$ is open and $G$ is $C$-solid, then $f$ is $C$-solid.
\end{proposition}
\begin{proof}
In the proof we say ``solid'' instead of ``$C$-solid''. Let (*) be the
property of a sheaf $G$ that any open $f:F\sr G$ is solid. If $G$ is
solid, $G$ sits at the end of a chain $\emptyset=G_0\sr G_1\sr\dots\sr
G_n=G$ with each $G_i\sr G_{i+1}$ push out of some $h_U\sr h_X$ for
$U\sr X$ in $C$. We prove by induction on $i$ that $G_i$ satisfies
(*).

For $i=1$, $G_1=h_X$ is representable and $\emptyset\sr X$ is in
$C$. If $f:F\sr G_1$ is open, it is of the form $h_U\sr h_X$ for $U$
open in $X$, hence solid by \ref{cond1}(1). It remains to check that
if in a cocartesian square
\begin{equation}
\llabel{8.3.1}
\begin{CD}
h_U @>>> h_X\\
@VVV @VVV\\
G' @>>> G
\end{CD}
\end{equation}
the sheaf $G'$ satisfies (*), so does $G$.  In (\ref{8.3.1}),
$h_U\sr h_X$ is a monomorphism and the square (\ref{8.3.1}) hence
cartesian as well as cocartesian.

Fix $f:F\sr G$ open, and take the pull-back of (\ref{8.3.1}) by
$f$. It is again a cartesian and cocartesian square and, $f$ being
open, it is of the form
\begin{equation}
\llabel{8.3.2}
\begin{CD}
h_V @>>> h_Y\\
@VVV @VVV\\
F' @>>> F
\end{CD}
\end{equation}
where $Y$ is open in $X$ and $V=U\cap Y$. The diagram
$$
\begin{CD}
F' @>>> F @>>> h_{Y/V}\\
@VVV @VVV @VVV\\
G' @>>> G @>>> h_{X/U}\\
@VVV @VVV @VVV\\
G'/F' @>>> G/F @>>> h_{X/(U\cup V)}
\end{CD}
$$
expresses $G/F$ as an extension of $h_{X/(U\cup Y)}$ by $G'/F'$ and
one concludes by  \ref{nrm} using (*) for $G'$ and the fact that
$U\cup Y\sr X$ is in $C$.
\end{proof}
\begin{proposition}
\llabel{9.6new}
The pull-back of a solid morphism $f$ by an open morphism $s$ is
solid. In particular, if $g:F\sr G$ is open and if $G$ is solid, so is
$F$.
\end{proposition}
\begin{proof}
Since the pull-back of an open morphism is open, it suffices to check
the proposition for $f$ a push-out of $h_U\subset h_X$ for $U$ open in
$X$:
$$
\begin{CD}
h_U @>>> h_X\\
@VVV @VVV\\
G' @>f>> G
\end{CD}
$$
Pulling back by $g$, we obtain a cocartesian square
$$
\begin{CD}
h_{U'} @>>> h_{X'}\\
@VVV @VVV\\
F' @>>> F
\end{CD}
$$
with $U'$ open in $U$ and $X'$ open in $X$. This shows that $F'\sr F$
is solid.
\end{proof}

Suppose now that we are given two clases $C$ and $C'$ of open
embeddings satisfying conditions \ref{cond1}. We define $C\times C'$
as the smallest class stable by \ref{cond1} containing the
$$(U\times V')\cup (U'\times V)\subset V\times V'$$
for $U\subset V$ in $C$ and $U'\subset V'$ in $C'$. 
\begin{example}\rm\llabel{ex11.9.2}
If $C$ is a class of all open embeddings and $C'$ is the class of
those $U'\subset V'$ for which $G$ acts freely outside $U'$, then
$C\times C'=C'$.
\end{example} 
\begin{proposition}
\llabel{11.2}
If the pointed sheaves $F$ and $F'$ are respectivley $C$ and
$C'$-solid, the $F\wedge F'$ is $C\times C'$-solid.
\end{proposition}
\begin{proof}
By assumption, $F$ is an iterated extension in the sense of \ref{nrm} of
pointed sheaves $h_{V_i/U_i}$ with $U_i\sr V_i$ in $C$. Similarly for
$F'$, with $U_j'\sr V_j'$ in $C'$. The smash product $F\wedge F'$ is
then an iterated extension of the 
$$h_{V_i/U_i}\wedge h_{V_j'/U'_j}=h_{V_i\times V_j'/((U_i\times
V_j')\cup(V_i\times U_j'))},$$
taken for instance in the lexicographical order, hence it is $C\times
C'$ solid.
\end{proof}

\begin{definition}
\llabel{d8.3} A morphism is called ind-solid relative to $C$ if it is
a filtering colimit of $C$-solid morphisms.
\end{definition}
\begin{ex}\rm\llabel{ex1}
We take $G$ to be the trivial group. A section on $Y$ of a push-out
$$
\begin{CD}
h_U @>>> h_X\\
@V\psi VV @VVV\\
F @>>> G
\end{CD}
$$
can be described as follows. For an open subset $V$ of $Y$ and a
section $\phi$ of $F$ on $V$ consider on the small Nisnevich site
$Y_{Nis}$ of $Y$ the presheaf $\Phi(V,\phi)$ which sends $a:W\sr Y$ to the
set of morphisms $f:W\sr X$ such that $f^{-1}(U)=a^{-1}(V)$ and
$\phi_{|a^{-1}(V)}=f^*(\psi)$. A section of $G$ on $Y$ is given by data:
\begin{enumerate}
\item an open subset $V$ of $Y$
\item a section $\phi$ of $F$ on $V$
\item a section of $i^*(a_{Nis}\Phi(V,\phi))$ on $Y-V$ where $i$ is
the closed embedding $Y-V\sr Y$ and $a_{Nis}$ denotes the associated
Nisnevich sheaf.
\end{enumerate}
\end{ex}
\begin{ex}\rm\llabel{ex2}
In the notations of \ref{ex1}, if $F$ is a sheaf for the etale
topology, so is $G$. For any $Y$, the $(V,\phi)$ as in (1),(2) above
form a sheaf for the etale topology. It hence suffices to prove that
for $(V,\phi)$ fixed, the datum (3) forms a sheaf for the etale
topology. This is checked by using the following criterion to check if
a Nisnevich sheaf is etale. For $y\in Y$, and for $L$ a finite
separable extension of $k_y$, let ${\cal O}^h_{L,y}$ be deduced by
``extension of the residue field'' from the henselization ${\cal
O}^h_{y}$ of $Y$ at $y$. The criterion is that $Spec(L)\mapsto
F(Spec({\cal O}^h_{L,y}))$ should be an etale sheaf on
$Spec(k_y)_{et}$.
\end{ex}
\begin{ex}\rm
It follows from \ref{ex1} and \ref{ex2} that if $f:F\sr G$ is ind
solid, and if $F$ is etale, then $G$ is etale. In particular, a solid
sheaf, as well as a pointed solid sheaf, are etale sheaves.
\end{ex}
\begin{remark}\rm
The same formalism of open and solid morphisms holds in the site of
all schemes of finite type over $S$ with the etale topology.
\end{remark}

\subsection{A criterion for preservation of local
equivalences}\llabel{15}
We work with pointed sheaves on $QP/G$. Our aim in this section is to
prove the following result
\begin{theorem}\llabel{15.1}
Let $\Phi$ be a functor from pointed sheaves to pointed sets. Suppose
that $\Phi$ commutes with all colimits, and that for any open
embedding $U\sr X$, $\Phi((h_U)_+)\sr \Phi((h_X)_+)$ is a
monomorphism. Then if $f:F_{\BB}\sr G_{\BB}$ is a local 
equivalence and if $F_n$ and $G_n$ are (pointed) ind-solid, then
$\Phi(f)$ is a weak equivalence.
\end{theorem}
Suppose that 
$$
Q\,\,\,\,\,\,\,:\,\,\,\,\,\,\,
\begin{CD}
B @>>> Y\\
@VVV @VVV\\
A @>>> X
\end{CD}
$$
is an upper distinguished square. Adding a base point, we obtain
$Q_+$. The morphism $K_{Q_+}\sr X_+$ is then a local 
equivalence. Let us check that $\Phi(K_{Q_+})\sr \Phi(X_+)$ is a weak
equivalence. As $\Phi$ commutes with coproducts, this morphism is
deduced from the commutative square 
$$
\begin{CD}
\Phi((h_B)_+) @>>> \Phi((h_Y)_+)\\
@VVV @VVV\\
\Phi((h_A)_+) @>>> \Phi((h_X)_+)
\end{CD}
$$
by applying the same construction (\ref{14.1.3}). This square is
cocartesian because $Q$ is. The top horizontal line being a
monomorphism, it is homotopy cocartesian, and the claim follows. As
$\Phi$ commutes with colimits, this special case implies that more
generally one has
\begin{lemma}\llabel{15.2}
If $f$ is in the $\bdl$-closure of the $(K_Q)_+\sr X_+$ as above, then
$\Phi(f)$ is a weak equivalence.
\end{lemma}
\begin{piece}[Proof of \ref{15.1}]\llabel{15.7}\rm For any pointed sheaf
$F$, $R(F)\sr F$ is a local  equivalence. Indeed for any connected
$X$ in $QP/G$, $R(F)(X)\sr F(X)$ is a weak equivalence by
\ref{15.4}. It follows that for $f:F\sr G$ a local  equivalence,
$$
\begin{CD}
R(F) @>>> R(G)\\
@VVV @VVV\\
F @>>> G
\end{CD}
$$
is a commutative square of local  equivalences. By \ref{15.2} and
\ref{15.6}, $\Phi(R(f))$ is a weak equivalence. It remains to see that
$\Phi(R(F))\sr \Phi(F)$ is a weak equivalence - and the same for
$G$. For this it suffices to see that for a pointed ind-solid sheaf
$F$, $\Phi(R(F))\sr \Phi(F)$ is a weak equivalence. As $\Phi$ and $R$
commute with filtering colimits, the ind-solid reduces to solid, and
by the inductive definition of solid, it suffices to prove the
following lemma.
\end{piece}
\begin{lemma}\llabel{15.8}
Let $U\sr X$ be an open embedding. If in a cartesian square of pointed
sheaves 
$$
Q\,\,\,\,\,\,\,:\,\,\,\,\,\,\,
\begin{CD}
(h_U)_+ @>>> (h_X)_+\\
@VVV @VVV\\
F @>>> G
\end{CD}
$$ 
$F$ is such that $\Phi R(F)\sr \Phi(F)$ is a weak equivalence, the
same holds for $G$.
\end{lemma}
\begin{proof}
Consider the cocartesian square 
$$
Q'\,\,\,\,\,:\,\,\,\,\,
\begin{CD}
R((h_U)_+) @>>> R((h_X)_+)\\
@VVV @VVV\\
R(F) @>>> R
\end{CD}
$$
One can easily see that the top morphism is a monomorphism. It follows
that $Q'$ is point by point homotopy cocartesian, and $R\sr R(G)$ is a
local  equivalence. The functor $i^*i_*$ of \ref{15.4} transforms
a monomorphism into a coprojection of the form $A\sr A\vee ((\coprod
h_{U_i})_+)$. It follows that each $R_n$ is of the form $(\coprod
h_{U_i})_+$ and, by \ref{15.2} and \ref{15.6}, $\Phi(R)\sr
\Phi(R(G))$ is a weak equivalence. It remains to show that $\Phi(R)\sr
\Phi(G)$ is a weak equivalence. 

Let us apply $\Phi$ to the morphism of cocartesian squares $Q'\sr
Q$. By \ref{15.4} both $R((h_U)_+)\sr (h_U)_+$ and $R((h_X)_+)\sr
(h_X)_+$ are homotopy equivalences, and remain so by applying
$\Phi$. We assumed $\Phi R(F)\sr \Phi(F)$ to be a weak equivalence. As
$\Phi(Q')$ and $\Phi(Q)$ are cocartesian with a top morphism which is
a monomorphism (by the assumption on $\Phi$, for $Q$), it follows that
$\Phi(R)\sr \Phi(G)$ is a weak equivalence. Hence so is $\Phi(R(G))\sr
\Phi(G)$. 
\end{proof}

\section{Two functors}\llabel{11}

\subsection{The functor $X\mapsto X/G$}
One has a natural morphism of sites
$$\eta:(QP/G)_{Nis}\sr (QP)_{Nis}$$
given by the functor
$$\eta^f:QP\sr QP/G: X\mapsto (\mbox{\rm $X$ with the trivial
$G$-action})$$
Indeed, the functor $\eta^f$ commutes with finite limits and
transforms covering families into covering families. 

In particular the functor $\eta^f$ is continuous: if $F$ is a sheaf on
$(QP/G)_{Nis}$, the presheaf 
$$X\mapsto F(\mbox{\rm $X$ with the trivial
$G$-action})$$
is a sheaf on $(QP)_{Nis}$. The functor $\eta^f$ has a left adjoint
$\lambda^f:X\mapsto X/G$. As $\eta^f$ is continuous, the functor
$\lambda^f$ is cocontinuous, and the functor $\eta^*$ from sheaves on
$(QP)_{Nis}$ to sheaves on $(QP/G)_{Nis}$ is
$$F\mapsto (\mbox{\rm sheaf associated to the presheaf $X\mapsto
F(X/G)$})$$

\begin{proposition}
\llabel{10.3}
The cocontinuous functor $\lambda^f:X\mapsto X/G$ is also continuous,
that is, if $F$ is a sheaf on $(QP)_{Nis}$, the presheaf $X\mapsto
F(X/G)$ on $(QP/G)_{Nis}$ is a sheaf.
\end{proposition}
\begin{proof}
By  \ref{10.4} it is sufficient to test the sheaf property of
$X\mapsto F(X/G)$ for a covering of $X$ deduced by pull-back from a
Nisnevich covering $V_i\sr X/G$ of $X/G$. Passage to quotient commutes
with flat base change. Taking as base $X/G$, this gives that
$$X\times_{X/G} V_i\sr V_i$$
identifies $V_i$ with the quotient of $X\times_{X/G}V_i$ by
$G$. Similarly, if $V_{ij}=V_i\times_{X/G} V_j$, the quotient by $G$
of the pull-back to $X$ of $V_{ij}$ is $V_{ij}$ again. This reduces
the sheaf property of $X\mapsto F(X/G)$, for the covering of $X$ by
the $X\times_{X/G} V_i$, to the sheaf property of $F$ for the covering
$(V_i)$ of $X/G$.
\end{proof}
The functor $\lambda^f:X\mapsto X/G$ gives rise to a pair of adjoint
functors $(\lambda_*, \lambda^*)$ between the categories of presheaves
on $(QP/G)$ and $(QP)$, with $\lambda_*(F)$ being $X\mapsto
F(X/G)$. As $\lambda^f$ is continuous , it induces a similar pair of
adjoint functors between the categories of sheaves. This pair is 
$$(\eta_{\#}:=(\mbox{\rm associated sheaf})\circ \lambda^*,
\eta^*=\lambda_*),$$
so that one has a sequence of adjunctions
$(\eta_{\#},\eta^*,\eta_*)$. If $F$ on $(QP/G)$ is representable:
$F=h_X$, then $\eta_{\#}(F)=h_{X/G}$. In particular, $\eta_{\#}$
transforms the final sheaf $h_S$ on $(QP/G)_{Nis}$, also called
``point'', into the final sheaf on $(QP)_{Nis}$, and
$(\eta_{\#},\eta^*)$ is a pair of adjoint functors in the category of
pointed sheaves as well.
It is clear that $\eta_{\#}$ takes solid sheaves to solid sheaves. We
also have the following.
\begin{proposition}
\llabel{new3}
Let $F$ be a pointed sheaf solid with respect to open emeddings
$U\subset V$ of smooth schemes such that the action of $G$ on $V$ is
free outside $U$. Then $\eta_{\#}(F)$ is solid with respect to open
embeddings of smooth schemes.
\end{proposition}
\begin{proof}
If $V'$ is the open subset of $V$ where the action of
$G$ is free, then $U\cup V'=V$ and if $U':=U\cap V'$, a push-out of
$U\sr V$ is also a push-out of $U'\sr V'$: we gained that the action
is free everywhere. The next step is applying $\eta_{\#}$, from
pointed sheaves on $(QP/G)$ to pointed sheaves on $(QP)$. This functor
is a left adjoint, hence respects colimits and in particular
push-outs. It transforms $h_U$ to $h_{U/G}$, and in particular, for
$U=S$, the final object into the final object. To check that it
respects solidity it is hence sufficient to apply:
\begin{lemma}
\llabel{smsm}
If $G$ acts freely on $U$ smooth over $S$, then $U/G$ is smooth.
\end{lemma}
\begin{proof}
If $G$ is finite etale, for instance $S_n$, the case which most
interests us, this is clear, resulting from $U\sr U/G$ being etale. In
general one proceeds as follows. The assumption that $G$ acts freely
on $U$ implies that $U$ is a $G$-torsor over $U/G$. In particular,
$U\sr U/G$ is faithfully flat. As $U$ is flat over $S$, this forces
$U/G$ to be flat over $S$. To check smoothness of $U/G$ over $S$ it is
hence sufficient to check it geometric fiber by geometric fiber.  For
$\bar{s}$ a geometric point of $S$, smoothness of $(U/G)_{\bar{s}}$
amounts to regularity. As $U_{\bar{s}}$ is smooth over $\bar{s}$,
hence regular, and $U_{\bar{s}}\sr (U/G)_{\bar{s}}$ is faithfully
flat, this is \cite[??]{EGA4} (an application of Serre's cohomological
criterion for regularity).
\end{proof}
\end{proof}
\begin{proposition}
\llabel{new4}
The functor $\eta_{\#}$ respects local (resp. $\af$-) equivalences
between termwise ind-solid simplicial sheaves. 
\end{proposition}
\begin{proof}
Let $f:F\sr F'$ be a local equivalence between termwise ind-solid
simplicial sheaves on $QP/G$. To verify that $\eta_{\#}(f)$ is a local
equivalence it is sufficient to check that for any $X$ in $QP$ and
$x\in X$ the map 
$$\eta_{\#}(F)(Spec {\cal O}^n_{X,x})\sr \eta_{\#}(F')(Spec {\cal
O}^n_{X,x})$$
is a weak equivalence of simplicial sets. Since $\eta_{\#}$ is a left
adjoint, the functor
\begin{equation}\llabel{neq1}
F\mapsto \eta_{\#}(F)(Spec {\cal O}^n_{X,x})
\end{equation}
commutes with colimits. For an open embedding $U\sr V$ in $QP/G$,
$U/G\sr V/G$ is again an open embedding and we can apply to
(\ref{neq1}) Theorem \ref{15.1}.

Let $f:F\sr F'$ be an $\af$-equivalence. Consider the square
$$
\begin{CD}
R(F) @>R(f)>> R(F')\\
@VVV @VVV\\
F @>f>> F'
\end{CD}
$$
By the first part of proposition $\eta_{\#}$ takes the vertical maps
to local equivalences. Since $\eta_{\#}$ commutes with colimits, Lemma
\ref{new1} implies that $\eta_{\#}(R(f))$ is in the $\bdl$-closure of
the class which contains $\eta_{\#}((K_Q)_+\sr X_+)$ for $Q$ upper
distinguished and $\eta_{\#}((X\times\af)_+\sr X_+)$ for $X$ in
$QP/G$. By Theorem \ref{postponed} it suffice to prove that morphisms
of these two types are $\af$-equivalences. For morphisms of the first
type it follows from the first half of the proposition. For the
morphism of the second type it follows from the fact that morphisms
$\eta_{\#}((X\times\af)_+\sr X_+)$ and
$\eta_{\#}(X_+\stackrel{Id\times\{0\}}{\sr}(X\times\af)_+)$ are
mutually inverse $\af$-homotopy equivalences. 
\end{proof}
Define ${\bf L}\eta_{\#}:Ho_{\BB}\sr Ho_{\BB}$ (and similarly on
$Ho_{\af,\BB}$) setting
$${\bf L}\eta_{\#}(F):=\eta_{\#}(R(F))$$
where $R(F)$ is defined in \ref{15.4}. Proposition \ref{new4} shows
that ${\bf L}\eta_{\#}$ is well defined and that for a termwise
ind-solid $F$ one has ${\bf L}\eta_{\#}(F)\cong \eta_{\#}(F)$.

\subsection{The functor $X\mapsto X^{W}$}
As in Section \ref{10}, we fix $G$ and $S$. We also fix $W$ in $QP/G$
which is finite and flat over $S$.

For $F$ a presheaf on $QP/G$, we define $F^W$ to be the internal hom
object $\uu{Hom}(h_W,F)$. Its value on $U$ is $F(U\times_S W)$. If $F$
is a sheaf, so is $F^W$.
\begin{example}\rm\llabel{ex11.1}
Take $G$ and $W$ deduced from the finite group $S_n$ acting on
$\{1,\dots,n\}$ by permutations. In that case, if $F$ is represented
by $X$, with a trivial action of $S_n$, then $F^W$ is represented by
$X^n$, on which $S_n$ acts by permutation of the factors. 
\end{example}
\begin{remark}\rm\label{rm11.2}
If $F$ is representable (resp. and represented by $X$ smooth over $S$),
so is $F^W$. More precisely, if $F$ is represented by $X$ in $QP/G$,
consider the contravariant functor on $Sch/S$ of morphisms of schemes
from $W$ to $X$, that is the functor
$$U\mapsto Hom_U(W\times_S U, X\times_S U)$$
This functor is representable, represented by some $Y$
quasi-projective over $S$ (resp. and smooth). This $Y$ carries an
obvious action $\rho$ of $G$, and $(Y,\rho)$ in $QP/G$ represents
$F^W$. Proof: by attaching to a morphism $W\sr X$ its graph, one maps
the functor considered into the functor of finite subschemes of
$W\times_S X$, of the same rank as $W$, that is the functor
$$U\mapsto 
\left\{
\begin{array}{c}
\mbox{\rm subschemes of $(W\times_S X)\times_S U$ finite
and flat over $U$,}\\
\mbox{\rm with the same rank as $W\times_S U$ over $U$.}
\end{array} 
\right\}
$$
The later functor is represented by a quasi-projective scheme $Hilb$,
by the theory of Hilbert schemes. The condition that $\Gamma\subset
W\times_S X$ be the graph of a morphism from $W$ to $X$ is an open
condition. This means: let $\Gamma\subset (W\times_S X)\times_S U$ be
a subscheme finite and flat over $U$. There is $U'$ open in $U$ such
that for any base change $V\sr U$, the pull-back $\Gamma_V$ of
$\Gamma$ is the graph of some $V$-morphism from $W\times_S V$ to
$W\times_S X$ if and only if $V$ maps to $U'$. This gives the
existence of the required $Y$, and that it is open in $Hilb$. If $X$
is smooth the smoothness of $Y$ follows from the infinitesimal lifting
criterion. The quasi-projectivity follows from that of $Hilb$. On the
functors represented, the action $g(y)=gyg^{-1}$ of $G$ is clear. For
$T$ in $QP/G$, one has
$$Hom_{QP/G}(T,Y)=Hom_{G}(T,\uu{Hom}(W,X))=Hom_{G}(T\times_S W,X)=$$
$$=Hom_{QP/G}(T\times_S W,X)=F^W(T)$$
\end{remark}
Let $C$ be a class of open embeddings in $(QP/G)_{Nis}$. We will
simply say ``solid'' for ``$C$-solid''.
\begin{theorem}\llabel{11.3}
If $F$ is a solid sheaf on $(QP/G)_{Nis}$, so is $F^W$.
\end{theorem}
If a morphism of sheaves $A\sr F$ is open, i.e. relatively
representable by open (equivariant) embeddings, there is a natural
sequence of sheaves intermediate between $A^W$ and $F^W$. In the case
considered in  \ref{ex11.1}, and for $h_U\sr h_X$, they are
represented by the open equivariant subschemes $(X,U)^n_k$ of $X^n$
consisting of those n-uples $(x_1,\dots,x_n)$ for which at least $k$
of the $x_i$ are in $U$. The formal definition is as follows. 

A section of $F^W$ over $T$ is a section $s$ of $F$ over $W\times_S
T$. Let $U(s)$ be the equivariant open subscheme of $W\times_S
T$ on which $s$ is in $A$. The sheaf $(F,A)^W_k$ is the subsheaf of
$F^W$ consisting of those $s$ such that all fibers $U(s)_t$ of $U(s)$
over $T$ are of degree at least $k$. The condition that the fiber at
$k$ be of degree $\ge k$ is open in $t$, and it follows that the
inclusion of $(F,A)^W_k$ in $F^W$ is open. For $k=0$, $(F,A)^W_k$ is
simply $F^W$. For $k$ large, it is $A^W$.
\begin{lemma}
\llabel{11.4}
Suppose that $A\sr F$ is deduced by push-out from an open map $B\sr
G$, so that we have a cocartesian square
\begin{equation}
\llabel{11.4.1}
\begin{CD}
B @>>> G\\
@VVV @VVV\\
A @>>> F
\end{CD}
\end{equation}
Then, for each $k$, the cartesian square
\begin{equation}
\llabel{11.4.2}
\begin{CD}
({\rm fiber\,\,\,\, product})@>>> (A\coprod G, A)^W_k\\
@VVV @VVV\\
(F,A)^W_{k+1} @>>> (F,A)^W_k
\end{CD}
\end{equation}
is cocartesian as well.
\end{lemma}
\begin{proof}
The site $(QP/G)_{Nis}$ has enough points: as a consequence of 
\ref{10.4}, for each $X$ in $QP/G$ and $x\in X/G$, the functor
$$F\mapsto colim\,\,\,\, F(X\times_{X/G} V),$$
the limit being taken over the Nisnevich neighborhoods of $x$ in
$X/G$, is a point (= a fiber functor). The class of all such functors
is clearly conservative. Such a functor depends only on
$Y:=X\times_{X/G} (X/G)^h_x$, where $(X/G)^h_x$ is the henselization
of $X/G$ at $x$, and $Y$ can be any equivariant $S$-scheme which is a
finite disjoint union of local heselian schemes essentially of finite
type over $S$, and for which $Y/G$ is local. We call such a scheme
$G$-local henselian, and write $F\mapsto F(Y)$ for the corresponding
fiber functor.

We will show that (\ref{11.4.2}) becomes cocartesian after
application of any of the fiber functors $F\mapsto F(Y)$ defined
above. It suffices to show that for any $s$ in $(F,A)^W_k(Y)$, the
fiber of (\ref{11.4.2})$(Y)$ above $s$ is cocartesian in $Set$. This
fiber is of the form
$$
\begin{CD}
K\times L @>>> K\\
@VVV @VVV\\
L @>>> \{s\}
\end{CD}
$$
and such a square is cocartesian if and only if whenever $K$ or $L$ is
empty, the other is reduced to one element. Here, we also know that
$L\sr \{s\}$ is a injective. It hence suffice to check that if $L$ is
empty, then $K$ is reduced to one element. Fix $s$ in $(F,A)^W_k(Y)$,
and view it as a section of $F$ over $W\times_S Y$. Let $U\subset
W\times_S Y$ be the open equivariant subset where it is in $A$. The
assumption that $s$ be in $(F,A)^W_k$ means that the degree of the
fiber $U_g$ of $U\sr Y$ at a closed point $y$ of $Y$ is at least
$k$. By $G$-equivariance of $U$, this degree is independent of the
chosen $y$. We have to show that if $s$ is not in $(F,A)^W_{k+1}(Y)$,
that is if this degree is exactly $k$, then $s$ is the image of a
unique element of $(A\coprod G, A)^W_k$.

The scheme $W\times_S Y$ is a disjoint union of $G$-local henselian
schemes $(W\times_S Y)_i$. By assumption, (\ref{11.4.1})$((W\times_S
Y)_i)$ is cocartesian, hense if $s$ is not in $A$ on $(W\times_S
Y)_i$, then on $(W\times_S Y)_i$ it comes from a unique $\tilde{s}_i$
in $G$. Let $(W\times_S Y)'$ be the union of those $(W\times_S Y)_i$
on which $s$ is in $A$, and $(W\times_S Y)''$ be the union of
$(W\times_S Y)_i$ on which it is not. That $s$ is in $(F,A)^W_k$ but
not in $(F,A)^W_{k+1}$, means that $(W\times_S Y)'$ is of degree $d=k$
over $Y$. On $(W\times_S Y)'$, $s$ is in $A$.  On $(W\times_S Y)''$,
it comes from a unique $\tilde{s}$ in $G$. The section
$$s_1:=(\mbox{\rm $s$ in $A$ on $(W\times_S Y)'$, $\tilde{s}$ on
$(W\times_S Y)''$})$$ 
of $A\coprod G$ over $W\times_S Y$ is a section of $(A\coprod G,
A)^W_k$ on $Y$ lifting $s$. It is the unique such lifting: any other
lifting $s_2$, viewed as a section of $A\coprod G$ on $W\times_S Y$,
can be in $A$ at most on $(W\times_S Y)'$, hence must be in $A$ on the
whole of $(W\times_S Y)'$ which has just the required degree over
$Y$. This determines $s_2$ uniquely on $(W\times_S Y)'$, where it is
in $A$, as well as on  $(W\times_S Y)''$, where it is the unique
lifting of $s$ to $G$.
\end{proof}
{\bf Proof of  \ref{11.3}:} The induction which works to
prove \ref{11.3} is the following. As $F$ is solid, it sits at the
end of a sequence 
$$\emptyset\sr F_1\sr\dots\sr F_n=F$$
where each $F_i\sr F_{i+1}$ is a push-out of some open embedding in
$QP/G$. We prove by induction on $i$ that for any $Y$, $(F_i\coprod
h_Y)^W$ is solid. For $i=1$, $F_1$ is representable, hence so is
$(F_1\coprod h_Y)^W$ (\ref{rm11.2}). A fortiori, $(F_1\coprod
h_Y)^W$ is solid. For the induction step one applies the following
lemma to $F_i\coprod h_Y\sr F_{i+1}\coprod h_Y$.
\begin{lemma}
\llabel{11.5}
Let 
$$
\begin{CD}
h_U @>>> h_X\\
@VVV @VVV\\
F @>>> G
\end{CD}
$$
be a cocartesian square with $U$ open in $X$. Suppose that for any
$Z$, $(F\coprod h_Z)^W$ is solid. Then $F^W\sr G^W$ is solid.
\end{lemma}
\begin{proof}
As $F\sr G$ is open by \ref{8.1.3}, the $(G,F)^W_j$ are defined. It
suffices to prove that for each $j$, the open morphism
$(G,F)^W_{j+1}\sr (G,F)^W_j$ is solid. 

By  \ref{11.4}, this morphism sits in a cartesian and
cocartesian square
\begin{equation}
\begin{CD}
({\rm fiber\,\,\,\, product})@>[2]>> (F\coprod h_X, F)^W_k\\
@VVV @VVV\\
(G,F)^W_{j+1} @>[1]>> (G,F)^W_j
\end{CD}
\end{equation}
By assumption, $(F\coprod h_X)^W$ is solid. It follows that $(F\coprod
h_X)^W_k$ is solid too (apply  \ref{9.6new} to the open
morphism $(F\coprod h_X)^W_k\sr (F\coprod h_X)^W$). As $[1]$ is open,
so is $[2]$, and by  \ref{8.3}, $[2]$ is solid. The map
$[1]$ is then solid as a push-out of a solid map.
\end{proof}
\begin{example}\rm It is not always true that if $f:A\sr B$ is a solid
morphism, so is $f^W$. Take $G$ the trivial group and $W$ two points
(i.e. $S\coprod S$). Then $F^W=F^2$. For any sheaf $F$, the inclusion
$f$ of $F$ in $F\coprod pt$ is solid (deduced by push-out from
$\emptyset\sr pt$), and applying $(-)^W$, we obtain $F^2\sr F^2\coprod
F\coprod F\coprod pt$. Pulling back by the natural map from $F$ to one
of the summands $F$ (an open map), we see that if $f^W$ is solid, so
is $F$.
\end{example}
\begin{cor}
\llabel{c11.6} If $f:F\sr G$ is open and $G$ solid, then $f^W:F^W\sr
G^W$ is solid. In particular, if $G$ is pointed solid, so is $G^W$.
\end{cor}
\begin{proof}
$f^W$ is open and one applies  \ref{11.3} and 
\ref{8.3}.
\end{proof}
We now define for pointed sheaves on $(QP/G)_{Nis}$ a ``smash''
variant of the construction $F\mapsto F^W$. If we assume that the
marked point $pt\sr F$ is open, it is
$$F^{\wedge W}:=F^W/(F,pt)^W_1$$
that is $F^W$ with $(F,pt)^W_1$ contracted to the new base point. This
definition can be repeated for any pointed sheaf if, for any
monomorphism $A\sr F$, we define$(F,A)^W_1$ to be the following
subsheaf of $F^W$: a section $s$ of $F^W(X)=F(X\times W)$ is in
$(F,A)^W_1$ if for any non empty $X'\sr X$, there exists a
commutative diagram 
$$
\begin{CD}
X'' @>>> X\times W\\
@VVV @VVV\\
X' @>>> X
\end{CD}
$$
with $X''$ non empty and $s$ in $A$ on $X''$. 
\begin{example}
\llabel{ex11.9}\rm Under the assumptions of \ref{ex11.1}, if $U$ is
open in $X$ with complement $Z$ and if $F=h_X/h_U$, then $F^{\wedge
W}$ is $h_{X^n}/h_{X^n-Z^n}$, where $X^n$ has the natural action of
$S_n$. In particular, if $F$ is the Thom space of a vector bundle $E$
over $Y$ (that is, $h_E$ with $h_{E-s_0(Y)}$ contracted to the base
point), then $F^{\wedge W}$ is the Thom space of the $S_n$-equivariant
vector bundle $\oplus pr^*_i E$ on $Y^n$.
\end{example}
\begin{proposition}\llabel{5.1.9}
If a pointed sheaf $F$ is pointed solid, that is if $pt\sr F$ is
solid, then so is $F^{\wedge W}$.
\end{proposition}
\begin{proof}
As $F^W$ is solid, the open morphism $(F,pt)^W_1\sr F^W$ is solid too
(\ref{8.3}), and $pt\sr F^{\wedge W}$ is deduced from it
by push-out.
\end{proof}
The definition of $F^{\wedge W}$ immediately implies the following:
\begin{lemma}
\llabel{new7}
Let $Y$ be a $G$-local henselian scheme. Then 
$$F^{\wedge W}(Y)=\bigwedge_{i} F((W\times Y)_i)$$
where $(W\times Y)_i$ are $G$-local henselian schemes such that
$$W\times Y=\coprod_i (W\times Y)_i$$
\end{lemma}
\begin{proposition}
\llabel{new5}
The functor $F^{\wedge W}$ respects local and $\af$-equivalences.
\end{proposition}
\begin{proof}
Let $f:F\sr H$ be a local equivalence. To check that $F^{W}\sr
H^{W}$ is a local equivalence it is enough to show that for any
$G$-local henselian $Y$, the map 
$$F^{W}(Y)=F(Y\times W)\sr H(Y\times W)=H^{W}(Y)$$
is a weak equivalence of simplicial sets. This follows from Lemma
\ref{new7}.

Let $f$ be an $\af$-equivalence. By the first part it is sufficient to
show that $R(f)^W:R(F)^W\sr R(H)^W$  is an $\af$-equivalence. We sue
the characterization of $\af$-equivalences given in Theorem
\ref{conv}. Since $F\mapsto F^{\wedge W}$ commutes with filtering
colimits and preserves local equivalences it suffices to check that it
takes $\af$-homotopy equivalences to $\af$-homotopy equivalences. This
is seen using the natural map 
$$F^{\wedge W}\wedge (h_{\af})_+\sr (F\wedge (h_{\af})_+)^{\wedge W}.$$
\end{proof}

\bibliography{alggeom}
\bibliographystyle{plain}
\end{document}